\documentclass[a4paper,10pt]{amsart}
\usepackage{amsmath,amssymb}
\usepackage{amscd}
\newtheorem{thm}{Theorem}[section]
\newtheorem{prop}[thm]{Proposition}
\newtheorem{lemma}[thm]{Lemma}
\newtheorem{cor}[thm]{Corollary}
\newtheorem{exam}[thm]{Example}
\theoremstyle{remark}
\newtheorem{remark}[thm]{Remark}
\newcommand{\id}{{\rm{id}}}
\newcommand{\Ad}{{\rm{Ad}}}

\newcommand{\BN}{\mathbf N}

\newcommand{\BC}{\mathbf C}

\newcommand{\BK}{\mathbf K}
\newcommand{\BB}{\mathbf B}

\newcommand{\la}{\langle}
\newcommand{\ra}{\rangle}

\newcommand{\CalL}{{\mathcal{L}}}
\newcommand{\CalE}{{\mathcal{E}}}
\newcommand{\CalH}{{\mathcal{H}}}
\newcommand{\CalK}{{\mathcal{K}}}
\newcommand{\CallN}{{\mathcal{N}}}
\newtheorem{Def}{Definition}[section]

\title{Strong Morita equivalence for completely positive linear maps on $C^*$-algebras}
\author{Kazunori Kodaka}
\address{Department of Mathematical Sciences, Faculty of Science, Ryukyu University, 
\endgraf
Nishihara-cho, Okinawa, 903-0213, Japan}
\address{\sl{E-mail address}: \rm{kodaka@math.u-ryukyu.ac.jp}}


\begin{document}
%\maketitle
\begin{abstract}
We will introduce the notion of strong Morita equivalence for completely positive linear maps
and study its basic properties. 
Also, we will discuss the relation between strong Morita equivalence for
bounded $C^*$-bimodule linear maps and strong Morita equivalence for completely positive linear maps.
Furthermore, we will show that if two unital $C^*$-algebras are strongly Morita equivalent, then
there is a $1-1$ correspondence between the two sets of all strong Morita equivalence classes
of completely positive linear maps on the two unital $C^*$-algebras and we will show that the
corresponding two classes of the completely positive linear maps are also strongly Morita equivalent.
\end{abstract}

\maketitle

\section{Introduction}\label{sec:intro} In the previous paper \cite {Kodaka:Picard3}, we introduced the notions of
strong Morita equivalence for bounded $C^*$-bimodule linear maps and the Picard group of
a bounded $C^*$-bimodule linear map.
\par
In this paper, we will introduce the notion of strong Morita equivalence for completely positive
linear maps applying its minimal Stinespring representation and \cite [Definition 2.1]{ER:multiplier} introduced
by Echterhoff and Raeburn. 
\par
To do this, following \cite [Definition 2.1]{ER:multiplier}, in Section \ref{sec:MR},
we will introduce the notion of strong Morita
equivalence for non-degenerate representations of $C^*$-algebras and we will discuss its basic properties.
\par
In Section \ref{sec:DP}, we will define strong Morita equivalence for completely positive linear maps and
discuss its basic properties.
\par
In Section \ref{sec:Rel}, we will consider inclusions of $C^*$algebras $A\subset C$ with $\overline{AC}=C$ and
conditional expectations. Conditional expectations are regarded as bounded $C^*$-bimodule linear
maps. Also, they are regarded as completely positive linear maps. We will discuss the relation between
strong Morita equivalence for bounded $C^*$-bimodule linear maps and strong Morita equivalence
for completely positive linear maps.
\par
In Section \ref{sec:Co}, we will consider two unital $C^*$-algebras $A$ and $B$, which are strongly Morita
equivalent. We will construct a $1-1$ correspondence between the set of all strong Morita equivalence classes of
completely positive linear maps from $A$ to the $C^*$-algebra of all bounded linear operators on a
Hilbert space and the set of all strong Morita equivalence classes of completely positive linear maps from
$B$ to the $C^*$-algebra of all bounded linear operators on a Hilbert space.

\section{Preliminaries}\label{sec:pre} Let $A$ be a $C^*$-algebra. We denote by $\id_A$
the identity map on $A$. When $A$ is unital, we denote by $1_A$ the unit element in $A$.
We simply them by $\id$ and $1$, respectively if no confusion arises.
\par
For each $n\in\BN$, let $M_n (A)$ be the $n\times n$-matrix algebra over $A$. We identify $M_n (A)$ with
$A\otimes M_n (\BC)$. Let $I_n $ be the unit element in $M_n (\BC)$. For each $a\in M_n (A)$, we denote
by $a_{ij}$, the $i\times j$-entry of the matrix $a$.
\par
Let $M(A)$ be the multiplier $C^*$-algebra of $A$ and for any automorphism $\alpha$
of $A$, let $\underline{\alpha}$ be the automorphism of $M(A)$ extending $\alpha$ to $M(A)$, which is defined in
Jensen and Thomsen \cite [Corollary 1.1.15]{JT:KK}.
\par
Let $A$ and $B$ be $C^*$-algebras and $X$ an $A-B$-equivalence
bimodule. We denote its left $A$-action and right $B$-action on $X$ by $a\cdot x$ and $x\cdot b$ for any
$a\in A$, $b\in B$, $x\in X$, respectively. Let $\widetilde{X}$ be the dual $B-A$-equivalence
bimodule of $X$ and let $\widetilde{x}$ denote the element in $\widetilde{X}$ associated to an element $x\in X$.
\par
For Hilbert spaces $\CalH$ and $\CalK$, let $\BB(\CalK, \CalH)$ be the space of all bounded linear operators
from $\CalK$ to $\CalH$ and if $\CalH=\CalK$, we denote $\BB(\CalH, \CalH)$ by $\BB(\CalH)$.
For a Hilbert space $\CalH$, we denote by $\la \cdot , \cdot \ra_{\CalH}$ the inner product of $\CalH$.
\par
Let $\BK$ be the $C^*$-algebra of all compact operators on a countably infinite dimensional Hilbert space
$\CalH_0$. Let $\{\epsilon_i \}_{i=1}^{\infty}$ be an orthogonal basis of $\CalH_0$ and $\{e_{ij}\}_{i,j=1}^{\infty}$
the system of matrix units of $\BK$ with respect to $\{\epsilon_i \}_{i=1}^{\infty}$.

\section{Definition and properties of Strong Morita equivalence
for non-degenerate representations of $C^*$-algebras}\label{sec:MR}
Following \cite[Definition 2.1]{ER:multiplier}, we give the definition of a representation of
equivalence bimodule.
\begin{Def}\label{Def:MR1} Let $A$ and $B$ be $C^*$-algebras. A
\sl
representation
\rm
of an $A-B$-equivalence bimodule $X$ on the pair of Hilbert spaces $(\CalH, \CalK)$ is a triple
$(\pi_A, \pi_X, \pi_B)$ consisting of non-degenerate representations $\pi_A : A\to\BB(\CalH)$, 
$\pi_B : B\to\BB(\CalK)$ and a linear map $\pi_X : X \to\BB(\CalK, \CalH)$ satisfying the following:
for any $a\in A$, $b\in B$, $x, y\in X$,
\newline
(1) $\pi_X (x)\pi_X(y)^* =\pi_A ({}_A \la x, y \ra)$,
\newline
(2) $\pi_X (x)^* \pi_X(y) =\pi_B ( \la x, y \ra_B )$,
\newline
(3) $\pi_X (a\cdot x\cdot b)=\pi_A (a)\pi_X (x)\pi_B (b)$.
\end{Def}

Let $(\pi_A , \CalH)$ and $(\pi_B , \CalK)$ be non-degenerate representations of $C^*$-algebras
$A$ and $B$, respectively.

\begin{Def}\label{Def:MR2} The non-degenerate representation $(\pi_A, \CalH)$ of $A$ is
\sl
strongly Morita equivalent
\rm
to the non-degenerate representation $(\pi_B, \CalK)$ of $B$ if there are an $A-B$-equivalence bimodule
$X$ and a linear map $\pi_X : X \to\BB(\CalK, \CalH)$ such that
$(\pi_A , \pi_X , \pi_B)$ is a representation of $X$
on the pair of Hilbert spaces $(\CalH, \CalK)$.
\end{Def}

\begin{lemma}\label{lem:MR3} With the above notation, strong Morita equivalence for non-degenerate
representations of $C^*$-algebras is equivalence relation.
\end{lemma}
\begin{proof} Let $(\pi_A , \CalH)$ be a non-degenerate representation of a $C^*$-algebra $A$.
We regard $A$ as the trivial $A-A$-equivalence bimodule in the usual way. We denote it by
$X_0$. Let $\pi_{X_0}$ be the linear map from $X_0$ to $\BB(\CalH)$ defined by $\pi_{X_0} (x)=\pi_A (x)$
for any $x\in X_0$. Then we can see that $\pi_{X_0}$ satisfies Conditions (1)-(3) in Definition \ref{Def:MR1}.
Hence $(\pi_A , \CalH)$ is strongly Morita equivalent to itself.
\par
Let $(\pi_B , \CalK)$ be a non-degenerate
representation of a $C^*$-algebra $B$ and we suppose that $(\pi_A , \CalH)$ is strongly Morita
equivalent to $(\pi_B , \CalK)$. Then there are an $A-B$-equivalence bimodule $X$ and a linear map
$\pi_X$ from $X$ to $\BB(\CalK, \CalH)$ satisfying Conditions (1)-(3) in Definition \ref{Def:MR1}.
Let $\widetilde{\pi_X}$ be the linear map from $\widetilde{X}$ to $\BB(\CalH, \CalK)$ defined by
$\widetilde{\pi_X}(\widetilde{x})=\pi_X (x)^*$ for any $x\in X$. Then we see that $\widetilde{\pi_X}$ satifies
Conditions (1)-(3) in Definition \ref{Def:MR1}. Thus, $(\pi_B , \CalK)$ is strongly Morita equivalent to
$(\pi_A, \CalH)$.
\par
Let $(\pi_C, \CalL)$ be a non-degenerate representation of a $C^*$-algebra $C$.
We suppose that $(\pi_A, \CalH)$ is strongly Morita equivalent to $(\pi_B , \CalK)$ with respect to
an $A-B$-equivalence bimodule $X$ and a linear map $\pi_X : X \to \BB(\CalK, \CalH)$ such that
$(\pi_A , \pi_X , \pi_B)$ is a representation of $X$ on the pair of Hilbert spaces $(\CalH, \CalK)$.
Also, we suppose that $(\pi_B , \CalK)$ is strongly Morita equivalent to $(\pi_C , \CalL)$ with respect to
a $B-C$-equivalence bimodule $Y$ and a linear map $\pi_Y : Y \to \BB(\CalL , \CalK)$ such that
$(\pi_B , \pi_Y , \pi_C)$ is a representation of $Y$ on the pair of Hilbert spaces $(\CalK, \CalL)$.
Let $\pi_{X\otimes_B Y}$ be the linear map from $X\otimes_B Y$ to $\BB(\CalL , \CalH)$ defined by
$$
\pi_{X\otimes Y}(x\otimes y)=\pi_X (x)\pi_Y (y)
$$
for any $x\in X$, $y\in Y$. Then by routine computations, we can see that $\pi_{X\otimes_B Y}$ satisfies
Conditions (1)-(3) in Definition \ref{Def:MR1}. Thus $(\pi_A , \pi_{X\otimes_B Y} , \pi_C)$ is a representation of
$X\otimes Y$ on a pair of Hilbert spaces $(\CalH , \CalL)$ and $(\pi_A , \CalH)$ is strongly Morita equivalent to
$(\pi_C , \CalL)$. Therefore, we obtain the conclusion.
\end{proof}

\begin{lemma}\label{lem:MR3-2} Let $(\pi_1 , \CalH_1)$ and $(\pi_2 , \CalH_2 )$ be non-degenerate representations
of a $C^*$-algebra $A$. If $(\pi_1 , \CalH_1)$ and $(\pi_2 , \CalH_2 )$ are unitarily equivalent, they are strongly
Morita equivalent.
\end{lemma}
\begin{proof} Since $(\pi_1 , \CalH_1)$ and $(\pi_2 , \CalH_2 )$ are unitarily equivalent,
there is an isometry $u$ from $\CalH_1$ onto $\CalH_2$ such that $\pi_2 =\Ad(u)\circ\pi_1$.
Let $X_0$ be the trivial $A-A$-equivalence bimodule defined in the proof of
Lemma \ref{lem:MR3}. Let $\pi_{X_0}$ be the linear map from $X_0$ to $\BB(\CalH_2 , \CalH_1 )$ defined by
$\pi_{X_0}(x)=\pi_1 (x)u^*$. Then by easy computations, we can see that $(\pi_1 , \pi_{X_0} , \pi_2 )$ is
a representation of $X_0$ on the pair of Hilbert spaces $(\CalH_1 , \CalH_2 )$. Hence
$(\pi_1 , \CalH_1)$ and $(\pi_2 , \CalH_2 )$ are strongly Morita equivalent.
\end{proof}

Let $(\pi_A , \CalH)$ and $(\pi_B , \CalK )$ be non-degenerate representations of $C^*$-algebras $A$
and $B$, respectively. We suppose that $(\pi_A , \CalH)$ and $(\pi_B , \CalK )$ are strongly Morita
equivalent, that is, there are an $A-B$-equivalence bimodule $X$ and a linear map $\pi_X$ from $X$
to $\BB(\CalK, \CalH)$ such that $(\pi_A , \pi_X , \pi_B)$ is representation of $X$ on the pair of Hilbert spaces
$(\CalH, \CalK)$.
\par
Let $L_X $ be the linking $C^*$-algebra for $X$, that is,
$$
L_X =\{\begin{bmatrix} a & x \\
\widetilde{y} & b \end{bmatrix} \, | \, a\in A , \quad
b\in B, \quad x, y\in X \} .
$$
Let $\rho$ be the representation of $L_X$ on $\CalH\oplus\CalK$ defined by
$$
\rho(\begin{bmatrix} a & x \\
\widetilde{y} & b \end{bmatrix} )=\begin{bmatrix} \pi_A (a) & \pi_X (x) \\
\pi_X (y)^* & \pi_B (b) \end{bmatrix} 
$$
for any $a\in A$, $b\in B$, $x, y\in X$. Since $(\pi_A , \CalH)$ and $(\pi_B , \CalK)$
are non-degenerate, so is $(\rho, \CalH\oplus\CalK)$ by \cite [\S 2, Remark (3)]{ER:multiplier}.
Let $p=\begin{bmatrix} 1_{M(A)} & 0 \\
0 & 0 \end{bmatrix}$, $q=\begin{bmatrix} 0 & 0 \\
0 & 1_{M(B)} \end{bmatrix}$. Then $p$ and $q$ are projections in $M(L_X )$ with
$$
\overline{L_X pL_X}=L_X , \quad \overline{L_X q L_X}=L_X , \quad
pL_X p\cong A , \quad qL_X q \cong B
$$
as $C^*$-algebras, respectively. We identify $pL_X p$ and $qL_X q$ with $A$ and $B$, respectively.

\begin{lemma}\label{lem:MR4} With the above notation, $(\rho, \CalH\oplus\CalK)$ is
strongly Morita equivalent to $(\pi_A, \CalH)$ and $(\pi_B , \CalK)$.
\end{lemma}
\begin{proof} We have only to show that $(\rho, \CalH\oplus\CalK)$ is strongly Morita
equivalent to $(\pi_A , \CalH)$ by Lemma \ref{lem:MR3}. Let $Y=pL_X$.
Since we identify $pL_X p$ with $A$, $Y$ can be regarded
as an $A-L_X$-equivalence bimodule in the usual way. Let $\pi_Y$ be the linear map
from $Y$ to $\BB(\CalH\oplus\CalK , \CalH)$ defined by
$$
\pi_Y (\begin{bmatrix} a & x \\
0 & 0 \end{bmatrix} )=\begin{bmatrix} \pi_A (a) & \pi_X (x) \\
0 & 0 \end{bmatrix} 
$$
for any $a\in A$, $x\in X$, where we identify $pL_X $ with the space
$$
\{ \begin{bmatrix} a & x \\
0 & 0 \end{bmatrix} \, | \, a\in A ,\quad x\in X \} .
$$
Then by routine computations, we can see that $(\pi_A , \pi_Y , \rho)$ is a representation of $Y$ on
the pair of Hilbert spaces $(\CalH, \CalH\oplus\CalK)$. Therefore, $(\rho, \CalH\oplus\CalK)$ is strongly Morita
equivalent to $(\pi_A , \CalH)$.
\end{proof}

Since $(\rho, \CalH\oplus\CalK)$ is a non-degenerate representation of $L_X$, by Pedersen
\cite [Theorem 3.7.7]{Pedersen:auto},
there is a unique normal homomorphism $\rho ''$ of $L_X ''$ onto $\rho(L_X )''$, which extends $\rho$,
where $L_X ''$ is the enveloping von Neumann algebra of $L_X$. Also, by \cite [3.12.1]{Pedersen:auto}
we can regard $M(L_X )$ as a $C^*$-subalgebra of $L_X ''$. Furthermore, we see that $\underline{\rho}$ is
the restriction of $\rho''$ to $M(L_X )$. Hence $\underline{\rho}(p)$ and $\underline{\rho}(q)$ are the projections
in $\BB(\CalH\oplus\CalK)$ with their ranges are $\CalH\oplus 0$ and $0\oplus\CalK$, respectively.
That is, $P_{\CalH}=\underline{\rho}(p)$ and $P_{\CalK}=\underline{\rho}(q)$, where $P_{\CalH}$
and $P_{\CalK}$ are projections from $\CalH\oplus\CalK$ onto $\CalH\oplus 0$ and $0\oplus\CalK$,
respectively. Furthermore, we assume that
$A$ and $B$ are $\sigma$-unital stable $C^*$-algebras. Then in the same way as in the proof of
Brown, Green and Rieffel \cite [Theorem 3.4]{BGR:linking}, there is a partial isometry $w\in M(L_X )$ such that
$w^* w=p$, $ww^* =q$. Let $\theta$ be the map from $pL_X p$ to $qL_X q$ defined by
$$
\theta(\begin{bmatrix} a & 0 \\
0 & 0 \end{bmatrix})=w\begin{bmatrix} a & 0 \\
0 & 0 \end{bmatrix}w^*
$$
for any $a\in A$. Then $\theta$ is an isomorphism of $pL_X p$ onto $qL_X q$. Identifying $A$ and $B$ with
$pL_X p$ and $qL_X q$, respectively, we can regard $\theta$ as an isomorphism of $A$ onto $B$. Let
$W=\underline{\rho}(w)$. Then $W\in \BB(\CalH\oplus \CalK)$ and
$$
W^* W=\underline{\rho}(w^* w)=\underline{\rho}(p)=P_{\CalH} , \quad
WW^* =\underline{\rho}(ww^* )=\underline{\rho}(q)=P_{\CalK} .
$$

\begin{lemma}\label{lem:MR4-2} With the above notation, there is an isometry $\widetilde{W}$ from
$\CalH$ onto $\CalK$ such that
$$
\pi_B (\theta(a))=\widetilde{W}\pi_A (a)\widetilde{W}^*
$$
for any $a\in A$.
\end{lemma}
\begin{proof} Let $\widetilde{W}=W|_{\CalH}=WP_{\CalH}$. Then for any $\xi \in\CalH$,
$$
\widetilde{W}\xi=WP_{\CalH}\xi=W\xi =WW^* W\xi=P_{\CalK}W\xi\in\CalK .
$$
Hence $\widetilde{W}\in\BB(\CalH, \CalK)$. For any $\eta\in\CalK$,
$$
W^* \eta =W^* WW^* \eta =P_{\CalH}W^* \eta \in\CalH .
$$
Then
$$
\widetilde{W}W^* \eta =WW^* \eta =P_{\CalK}\eta =\eta .
$$
Thus $\widetilde{W}$ is surjective. Furthermore, for any $\xi_1 , \xi_2 \in \CalH$,
\begin{align*}
\la \widetilde{W}\xi_1 \, , \, \widetilde{W}\xi_2 \ra_{\CalK} & =\la W\xi_1 \, , \, W\xi_2 \ra_{\CalH\oplus\CalK}
=\la \xi_1\,  , \, W^* W\xi_2 \ra_{\CalH\oplus\CalK}
=\la \xi_1 \, , \, P_{\CalH}\xi_2 \ra_{\CalH\oplus\CalK} \\
& =\la \xi_1 \, , \, \xi_2 \ra_{\CalH} .
\end{align*}
Hence $\widetilde{W}$ is an isometry from $\CalH$ onto $\CalK$. Finally, for any $a\in A$,
\begin{align*}
\pi_B (\theta(a)) & =\pi_B (w\begin{bmatrix} a & 0 \\
0 & 0 \end{bmatrix}w^* )
=\pi_B (ww^*w \begin{bmatrix} a & 0 \\
0 & 0 \end{bmatrix}w^* ww^* ) \\
& =\pi_B (qw \begin{bmatrix} a & 0 \\
0 & 0 \end{bmatrix}w^* q )
=\rho(qw \begin{bmatrix} a & 0 \\
0 & 0 \end{bmatrix}w^* q ) \\
& =\underline{\rho}(qw )\begin{bmatrix} \pi_A (a) & 0 \\
0 & 0 \end{bmatrix}\underline{\rho}(w^* q )
=P_{\CalK}W \begin{bmatrix} \pi_A (a) & 0 \\
0 & 0 \end{bmatrix} W^* P_{\CalK} \\
& =WP_{\CalH}\pi_A (a)P_{\CalH}W^* 
=\widetilde{W}\pi_A (a)\widetilde{W}^* .
\end{align*} 
Therefore, we obtain the conclusion.
\end{proof}

Next, we will give an easy example of non-degenerate representations of $C^*$-algebras which are
strongly Morita equivalent.
\par
Let $(\pi, \CalH)$ be a non-degenerate representation of a $C^*$-algebra $A$.
Let $A^s =A\otimes\BK$ and let $\pi^s =\pi\otimes\id_{\BK}$ and $\CalH^s =\CalH\otimes\CalH_0$. Then
$(\pi^s ,\CalH^s )$ is a non-degenerate representation of $A^s$ on $\CalH^s$.

\begin{exam}\label{exam:MR5} With the above notation, $(\pi, \CalH)$ and $(\pi^s , \CalH^s )$ are
strongly Morita equivalent.
\end{exam}
\begin{proof} Let $X=A^s (1\otimes e_{11})$.
Then $X$ is an $A^s -A$-equivalence bimodule in the usual way, where we identify $A$ with
$(1_{M(A)}\otimes e_{11})(A\otimes\BK)(1_{M(A)}\otimes e_{11})$.
Let $\imath_{\CalH}$ be the linear map from $\CalH$ to $\CalH^s$ defined by $\imath_{\CalH}\xi=\xi\otimes\epsilon_1$
for any $\xi\in\CalH$. Then $\imath_{\CalH}$ is an isometry and
$$
\imath_{\CalH}^* (\xi\otimes\epsilon_i )=\begin{cases} \xi & \text{if $i=1$} \\
0 & \text{if $i\geqq 2$} \end{cases}
$$
for any $\xi\in\CalH$ by routine computations. Thus $\imath_{\CalH}\imath_{\CalH}^* =1\otimes e_{11}$
and $\imath_{\CalH}^* \imath_{\CalH}=\id_{\CalH}$ on $\CalH$. Let $\pi_X$ be the linear map from $X$
to $\BB(\CalH, \CalH^s )$ defined by
$$
\pi_X (a(1_{M(A)}\otimes e_{11}))=\pi^s (a(1_{M(A)}\otimes e_{11}))\imath_{\CalH}
=\pi^s (a)(1_{M(A)}\otimes e_{11})\imath_{\CalH}
$$
for any $a\in A^s$. Then for any $a, b\in A^s$,
$$
\pi^s ({}_{A^s} \la a(1_{M(A)}\otimes e_{11}) \, , \, b(1_{M(A)}\otimes e_{11})\ra )
=\pi^s (a(1_{M(A)}\otimes e_{11})b^* ) .
$$
On the other hand,
\begin{align*}
\pi_X (a(1_{M(A)}\otimes e_{11}))\pi_X (b(1_{M(A)}\otimes e_{11}))^* 
& =\pi^s (a)(1\otimes e_{11})\imath_{\CalH}\imath_{\CalH}^* (1\otimes e_{11})\pi^s (b)^* \\
& =\pi^s (a)(1\otimes e_{11})\pi^s (b)^*
\end{align*}
since $\imath_{\CalH}i_{\CalH}^* =1\otimes e_{11}$. Thus
$$
\pi^s ({}_{A^s} \la a(1_{M(A)}\otimes e_{11}) \, , \, b(1_{M(A)}\otimes e_{11})\ra )
=\pi_X (a(1_{M(A)}\otimes e_{11}))\pi_X (b(1_{M(A)}\otimes e_{11}))^* 
$$
Also, for any $a, b\in A^s$,
$$
\pi(\la a(1\otimes e_{11}) \, , \, b(1\otimes e_{11})\ra_A )
=\pi^s ((1\otimes e_{11})a^* b(1\otimes e_{11}))
=(1\otimes e_{11})\pi^s (a^* b)(1\otimes e_{11}) .
$$
Since we identify $A$ with $(1\otimes e_{11})A^s (1\otimes e_{11})$, we regard
$\CalH$ as the closed subspace $\CalH\otimes\epsilon_1$ of $\CalH^s$. Hence
we can regard $(1\otimes e_{11})\pi^s (a^* b)(1\otimes e_{11})$ as an element
$\imath_{\CalH}^* (1\otimes e_{11})\pi^s (a^* b)(1\otimes e_{11})\imath_{\CalH}$ in $\BB(\CalH)$.
On the other hand,
\begin{align*}
\pi_X (a(1\otimes e_{11}))^* \pi_X (b(1\otimes e_{11}))
& =(\pi^s (a)(1\otimes e_{11})\imath_{\CalH})^* (\pi^s (b)(1\otimes e_{11})\imath_{\CalH}) \\
& =\imath_{\CalH}^* (1\otimes e_{11})\pi^s (a^* b)(1\otimes e_{11})\imath_{\CalH} .
\end{align*}
Thus
$$
\pi(\la a(1\otimes e_{11}) \, , \, b(1\otimes e_{11})\ra_A )
=\pi_X (a(1\otimes e_{11}))^* \pi_X (b(1\otimes e_{11}))
$$
for any $a, b\in A^s$. Furthermore, for any $a\in A^s$, $b\in A$, $x\in A^s$,
\begin{align*}
\pi_X (a\cdot x(1\otimes e_{11})\cdot b) & =\pi_X (ax(b\otimes e_{11}))
=\pi^s (ax(b\otimes e_{11}))\imath_{\CalH} \\
& =\pi^s (a)\pi^s (x)(1\otimes e_{11})(\pi(b)\otimes e_{11})\imath_{\CalH} .
\end{align*}
For any $\xi\in\CalH$,
$$
(\pi(b)\otimes e_{11})\imath_{\CalH}\xi =(\pi(b)\otimes e_{11})(\xi\otimes\epsilon_1 )
=\pi(b)\xi\otimes\epsilon_1 =\imath_{\CalH}\pi(b)\xi .
$$
Hence $(\pi(b)\otimes e_{11})\imath_{\CalH}=\imath_{\CalH}\pi(b)$. Thus
$$
\pi_X (a\cdot x(1\otimes e_{11})\cdot b)=\pi^s (a)\pi^s (x)(1\otimes e_{11})\imath_{\CalH}\pi(b)
=\pi^s (a)\pi_X (x(1\otimes e_{11}))\pi(b) .
$$
Therefore, $(\pi^s , \pi_X , \pi)$ is a representation of $X$ on the pair of spaces $(\CalH^s, \CalH)$, that is,
$(\pi^s , \CalH^s )$ and $(\pi, \CalH)$ are strongly Morita equivalent.
\end{proof}

\section{Definition and properties of strong Morita equivalence
for completely positive linear maps}\label{sec:DP} In this section, we define strong Morita equivalence
for completely positive linear maps from a $C^*$-algebra to a $C^*$-algebra of all bounded linear operators on
a Hilbert space.
\par
Let $\phi$ be a completely positive linear map from a $C^*$-algebra $A$ 
to $\BB(\CalH)$, where $\CalH$ is a Hilbert space $\CalH$.
Then by a Stinespring dilation theorem, there are a Hilbert space $\CalH_{\phi}$, a representation
$\pi_{\phi}$ of $A$ on $\CalH_{\phi}$ and $V_{\phi}\in\BB(\CalH, \CalH_{\phi})$ with $||V_{\phi}||=||\phi||$
such that
$$
\phi(a)=V_{\phi}^* \pi_{\phi}(a)V_{\phi}
$$
for any $a\in A$ and such that $\CalH_{\phi}=\overline{\pi_{\phi}(A)V_{\phi}\CalH}$.
For more information, see Blackadar \cite [II. 6.9.7]{Blackadar:operator}, Paulsen \cite [Theorem 4.1]{Paulsen:CB-maps}
and Stinespring \cite {Stine:positive}. We call the above $(\pi_{\phi}, V_{\phi}, \CalH_{\phi})$ a
\sl
minimal Stinespring representation
\rm
for $\phi$.
Furthermore, we can see that $(\pi_{\phi}, V_{\phi}, \CalH_{\phi})$ is unique in the sense of
\cite [Proposition 4.2]{Paulsen:CB-maps}, that is,

\begin{prop}\label{prop:DP1} $($\cite [Proposition 4.2]{Paulsen:CB-maps}$)$ With the above notation,
if $(\rho,W, \CalK)$ is another minimal Stinespring representation for $\phi$, then there is an
isometry $U$ from $\CalH_{\phi}$ onto $\CalK$ satisfying that $UV_{\phi}=W$ and that 
$U\pi_{\phi}(a)U^* =\rho(a)$ for any $a\in A$.
\end{prop}

\begin{remark}\label{rem:PD2} We note that a minimal Stinespring representation for $\phi$ is non-degenerate.
Indeed, let $(\pi_{\phi}, V_{\phi}, \CalH_{\phi})$ be a minimal Stinespring representation for $\phi$. Then
since $V_{\phi}\CalH\subset\CalH_{\phi}$,
$$
\CalH_{\phi}=\overline{\pi_{\phi}(A)V_{\phi}\CalH}\subset \overline{\pi_{\phi}(A)\CalH_{\phi}}
\subset \CalH_{\phi} .
$$
Thus $\CalH_{\phi}=\overline{\pi_{\phi}(A)\CalH_{\phi}}$, that is, $(\pi_{\phi}, \CalH_{\phi})$ is non-degenerate.
\end{remark}

Let $\phi$ and $\psi$ be completely positive linear maps from $C^*$-algebras $A$ and $B$ to $\BB(\CalH)$
and $\BB(\CalK)$, respectively, where $\CalH$ and $\CalK$ are Hilbert spaces.

\begin{Def}\label{Def:DP3} We say that $\phi$ is
\sl
strongly Morita equivalent
\rm
to $\psi$ if a minimal Stinespring representation for $\phi$ is strongly Morita equivalent to
that for $\psi$.
\end{Def}

\begin{prop}\label{prop:DP4} Strong Morita equivalence for completely positive linear maps on
$C^*$-algebras is equivalence relation.
\end{prop}
\begin{proof} This is immediate by Lemmas \ref{lem:MR3}, \ref{lem:MR3-2} and
Proposition \ref{prop:DP1}.
\end{proof}

Let $\phi$ and $\psi$ be completely positive linear maps from $C^*$-algebra $A$ and
$B$ to $\BB(\CalH)$ and $\BB(\CalK)$, respectively, which are strongly Morita equivalent.
Let $(\pi_{\phi}, V_{\phi}, \CalH_{\phi})$ and
$(\pi_{\psi}, V_{\psi}, \CalH_{\psi})$ be minimal Stinespring representations for them.
Then since $(\pi_{\phi}, \CalH_{\phi})$ and $(\pi_{\psi}, \CalH_{\psi})$ are strongly Morita
equivalent, there are an $A-B$-equivalence bimodule $X$ and a linear map $\pi_X$
from $X$ to $\BB(\CalH_{\psi}, \CalH_{\phi})$ such that $(\pi_{\phi}, \pi_X , \pi_{\psi})$ is
a representation of $X$ on the pair of Hilbert spaces $(\CalH_{\phi}, \CalH_{\psi})$.
Let $L_X$ be the linking $C^*$-algebra for $X$ and let $\rho$ be the representation of
$L_X$ on $\CalH_{\phi}\oplus\CalH_{\psi}$ induced by the representation $(\pi_{\phi}, \pi_X , \pi_{\psi})$
of $X$, which is defined in Section \ref{sec:MR}. Let $\tau$ be the completely positive linear map from $L_X$
to $\BB(\CalH\oplus\CalK)$ defined by
$$
\tau(\begin{bmatrix} a & x \\ 
\widetilde{y} & b \end{bmatrix})
=\begin{bmatrix} V_{\phi}^* & 0 \\ 
0 & V_{\psi}^* \end{bmatrix}
\rho(\begin{bmatrix} a & x \\
\widetilde{y} & b \end{bmatrix})
\begin{bmatrix} V_{\phi} & 0 \\
0 & V_{\psi} \end{bmatrix}
$$
for any $\begin{bmatrix} a & x \\
\widetilde{y} & b \end{bmatrix}\in L_X$.

\begin{lemma}\label{lem:DP5} With the above notation,
$(\rho, V_{\phi}\oplus V_{\psi}, \, \CalH_{\phi}\oplus\CalH_{\psi})$
is a minimal Stinespring representation for $\tau$.
\end{lemma}
\begin{proof} It suffices to show that $\rho(L_X )(V_{\phi}\oplus V_{\psi})(\CalH\oplus\CalK)$
is dense in $\CalH_{\phi}\oplus\CalH_{\psi}$. Indeed,
$$
\pi_{\phi}(A)V_{\phi}\CalH\oplus\pi_{\psi}(B)V_{\psi}\CalK\subset\rho(L_X )(V_{\phi}\oplus V_{\psi})(\CalH\oplus\CalK) .
$$
Since $\overline{\pi_{\phi}(A)V_{\phi}\CalH}=\CalH_{\phi}$ and $\overline{\pi_{\psi}(B)V_{\psi}\CalK}=\CalH_{\psi}$,
$$
\overline{\rho(L_X )(V_{\phi}\oplus V_{\psi})(\CalH\oplus\CalK)}=\CalH_{\phi}\oplus\CalH_{\psi} .
$$
Therefore, we obtain the conclusion.
\end{proof}

\begin{prop}\label{prop:DP6} Let $\phi$, $\psi$ and $\tau$ be as above. Then $\tau$ is
strongly Morita equivalent to $\phi$ and $\psi$ and
$$
\phi(a)=P_{\CalH} \tau(\begin{bmatrix} a & 0 \\
0 & 0 \end{bmatrix}) , \quad \psi(b)=P_{\CalK}\tau(\begin{bmatrix} 0 & 0 \\
0 & b \end{bmatrix}) ,
$$
for any $a\in A$ and $b\in B$, respectively,
where we identify $\CalH$ and $\CalK$ with $\CalH\oplus 0$ and $0\oplus \CalK$, respectively and $P_{\CalH}$
and $P_{\CalK}$ are projections from $\CalH\oplus\CalK$ onto $\CalH$ and $\CalK$, respectively.
\end{prop}
\begin{proof} This is immediate by Lemmas \ref{lem:MR4}, \ref{lem:DP5} and the definition of $\tau$.
\end{proof}

Furthermore, we assume that $A$ and $B$ are $\sigma$-unital stable $C^*$-algebras. Then
by the discussions before Lemma \ref {lem:MR4-2} and Lemma \ref{lem:MR4-2}, there are an isomorphism
$\theta$ of $A$ onto $B$ and an isometry $\widetilde{W}$ from $\CalH_{\phi}$ onto $\CalH_{\psi}$ such that
$$
\pi_{\psi}(\theta(a))=\widetilde{W}\pi_{\phi}(a)\widetilde{W}^*
$$
for any $a\in A$. By easy computations, we obtain that
$$
\psi(\theta(a))=V_{\psi}^* \pi_{\psi}(\theta(a))V_{\psi}=V_{\psi}^*\widetilde{W}\pi_{\phi}(a)\widetilde{W}^*V_{\psi} 
$$
for any $a\in A$.

\begin{prop}\label{prop:DP7} Let $A$ and $B$ be $\sigma$-unital stable $C^*$-algebras.
Let $\phi$ and $\psi$ be completely positive linear maps from $C^*$-algebra $A$ and
$B$ to $\BB(\CalH)$ and $\BB(\CalK)$, respectively. Let $(\pi_{\phi}, V_{\phi}, \CalH_{\phi})$ and
$(\pi_{\psi}, V_{\psi}, \CalH_{\psi})$ be minimal Stinespring representations for $\phi$ and $\psi$,
respectively. Then there are an isomorphism $\theta$ of $A$ onto $B$ and an isometry $\widetilde{W}$
from $\CalH_{\phi}$ onto $\CalH_{\psi}$ satisfying that $(\pi_{\phi}, \widetilde{W^*}V_{\psi}, \CalH_{\phi})$
is a minimal Stinespring representation for $\psi\circ\theta$.
\end{prop}
\begin{proof} Let $\widetilde{W}$ and $\theta$ be as before Lemma \ref {lem:MR4-2} and in the proof
of Lemma \ref {lem:MR4-2}. Then
it suffices to show that $(\pi_{\phi}, \widetilde{W}^* V_{\psi}, \CalH_{\pi})$
is a minimal Stinespring representation for $\psi\circ\theta$. Since $(\pi_{\psi}, V_{\psi}, \CalH_{\psi})$ is a
minimal Stinespring representation for $\psi$,
$$
\overline{\widetilde{W}\pi_{\phi}(A)\widetilde{W}^* V_{\psi}\CalK}=\overline{\pi_{\psi}(\theta(A))V_{\psi}\CalK}
=\overline{\pi_{\psi}(B)V_{\psi}\CalK}=\CalH_{\psi} .
$$
Since $\widetilde{W}$ is an isometry of $\CalH_{\phi}$ onto $\CalH_{\psi}$,
$$
\CalH_{\phi}=\widetilde{W}^* \CalH_{\psi}
=\widetilde{W}^* (\overline{\widetilde{W}\pi_{\phi}(A)\widetilde{W}^* V_{\psi}\CalK})
=\overline{\pi_{\phi}(A)\widetilde{W}^* V_{\psi}\CalK} .
$$
Thus, $(\pi_{\phi}, \widetilde{W}^* V_{\psi}, \CalH_{\phi})$ is a minimal Stinespring representation
for $\psi\circ\theta$.
\end{proof}

Let $\phi$ be a completely positive linear map from $A$ to $\BB(\CalH)$. We will show that $\phi$ and
$\phi\otimes\id_{\BK}$ are strongly Morita equivalent.

\begin{exam}\label{exam:DP8} With the above notation, $\phi$ and $\phi\otimes\id_{\BK}$ are
strongly Morita equivalent.
\end{exam}
\begin{proof} Let $(\pi_{\phi}, V_{\phi}, \CalH_{\phi})$ be a minimal Stinespring representation for
$\phi$. Then $(\pi_{\phi}^s , V_{\phi}\otimes 1_{\CalH_0}, \CalH_{\phi}^s )$ is a minimal Stinespring
representation for $\phi\otimes\id_{\BK}$, where $1_{\CalH_0}$ is the identity operator
on $\CalH_0$. Indeed,
for any $a\in A$, $k\in\BK$,
$$
(V_{\phi}\otimes 1_{\CalH^0})^* \pi_{\phi}^s (a\otimes k)(V_{\phi}\otimes 1_{\CalH^0})
=V_{\phi}^* \pi_{\phi}(a)V_{\phi}\otimes k
=(\phi\otimes\id_{\BK})(a\otimes k) .
$$
Also, 
$$\pi_{\phi}^s (A^s )(V_{\phi}\otimes 1_{\CalH^0})(\CalH_{\phi}^s )=\pi_{\phi}(A)V_{\phi}\CalH_{\phi}
\otimes\BK\CalH_0 .
$$
Since $\overline{\pi_{\phi}(A)V_{\phi}\CalH_{\phi}}=\CalH_{\phi}$, we can see that
$$
\overline{\pi_{\phi}^s (A^s )(V_{\phi}\otimes 1_{\CalH_0})(\CalH_{\phi}^s )}=\CalH_{\phi}^s .
$$
Since by Example \ref{exam:MR5}, $(\pi_{\phi}, \CalH_{\phi})$ and $(\pi_{\phi}^s , \CalH_{\phi}^s )$ are
strongly Morita equivalent, we obtain the conclusion.
\end{proof}

\section{Relation between strong Morita equivalence for bimodule linear maps
and strong Morita equivalence for completely positive linear maps}
\label{sec:Rel}
Let $A\subset C$ and $B\subset D$ be inclusions of $C^*$-algebras
with $\overline{AC}=C$ and $\overline{BD}=D$. Let $E^A$ and $E^B$ be conditional expectations from
$C$ and $D$ onto $A$ and $B$, respectively. We assume that $E^A$ and $E^B$
are strongly Morita equivalent with respect to a $C-D$-equivalence bimodule $Y$ and its closed subspace $X$
as bimodule linear maps
(See \cite [Definition 3.1]{Kodaka:Picard3}). We note that $A\subset C$ and $B\subset D$
are strongly Morita equivalent with respect to $Y$ and its
closed subspace $X$ (See \cite [Definition 2.1]{KT4:morita}) and that
$E^A$ and $E^B$ are completely positive linear maps from $C$ and $D$ onto $A$ and $B$,
respectively. It is natural that we consider whether $E^A$ and $E^B$ are
strongly Morita equivalent as completely positive linear maps.
\par
In this section, first, we will give the following result: We consider $E^A$ and $E^B$ which are strongly Morita equivalent
as bimodule linear maps.
For any non-degenerate representation $(\pi_B , \CalK_B)$ of $B$, there is a non-degenerate representation
$(\pi_A , \CalH_A )$ of $A$ such that $\pi_A \circ E^A$ and $\pi_B \circ E^B$ are strongly Morita
equivalent as completely positive linear maps. Also, we will consider its inverse direction.
\par
We will use the same notation as above.
We note that $\pi_B \circ E^B$ is a completely positive linear
map from $D$ to $\BB(\CalK_B )$. Hence by \cite [II. 6.9.7]{Blackadar:operator} or
\cite [Theorem 4.1]{Paulsen:CB-maps}, there is a minimal Stinespring representation $(\pi_D , V_D , \CalK_D )$
for $\pi_B \circ E^B$ such that
$$
(\pi_B \circ E^B )(d)=V_D^* \pi_D (d)V_D
$$
for all $d\in D$, where $V_D \in\BB (\CalK_B, \, \CalK_D)$. 
Modifying the proof of \cite [Theorem 4.1]{Paulsen:CB-maps}, we define $\CalK_D$ and $V_D$.
Let $D\odot\CalK_B$ be the algebraic tensor product of $D$ and $\CalK_B$. We define a symmetric bilinear
function $\la \cdot , \cdot \ra$ on $D\odot \CalK_B$ by setting
$$
\la d\otimes \xi \, , \, d_1 \otimes\eta \ra =\la \xi \, , \, (\pi_B \circ E^B)(d^* d_1 )\eta \ra_{\CalK_B}
$$
for any $d, d_1 \in D$, $\xi, \eta\in \CalK_B$ and extending linearly. Let $\CallN_D$ be the subspace of
$D\odot\CalK_B$ defined by
$$
\CallN_D =\{u\in D\odot \CalK_B \, | \, \la u , u \ra=0 \} .
$$
The induced bilinear form on the quotient space $(D\odot\CalK_B )/\CallN_D$ defined by
$$
\la u+\CallN_D \, , \, v+\CallN_D \ra=\la u \, , \, v \ra ,
$$
is an inner product, where $u, v\in D\odot \CalK_D$. We denote by $\CalK_D$
and $\la \cdot \, , \, \cdot \ra_{\CalK_D}$ the Hilbert space,
the completion of the innner product space $(D\odot \CalK_B )/\CallN_D$ and its inner product, respectively.
\par
We define $V_D : \CalK_B \to \CalK_D$ by setting
$$
V_D (\xi )=\lim_{\lambda}(d_{\lambda}\otimes\xi +\CallN_D )
$$
for any $\xi\in \CalK_B$, where $\{d_{\lambda}\}$ is an approximate unit of $D$ with
$0\leq d_{\lambda}$ and $||d_{\lambda}||\leq 1$ and the limit
is taken under the weak topology in $\CalK_D$. Also, we note that $\lim_{\lambda}(d_{\lambda}\otimes\xi+\CallN_D )$
is independent of the choice of an approximate unit of $D$. Furthermore,
$$
V_D^* (d\otimes\xi +\CallN_D )=(\pi_B \circ E^B)(d)\xi
$$
for any $d\in D$, $\xi\in\CalK_B$. Indeed, for any $d\in D$, $\xi, \eta\in \CalK_B$,
\begin{align*}
\la V_D^* (d\otimes\xi+\CallN_D ) \, , \, \eta \ra_{\CalK_B} & =
\la d\otimes \xi +\CallN_D \, , \, \lim_{\lambda}(d_{\lambda}\otimes\eta+\CallN_D ) \ra_{\CalK_D} \\
& =\lim_{\lambda}\la (\pi_B \circ E^B )(d_{\lambda}d)\xi \, , \, \eta \ra_{\CalK_B} \\
& =\la (\pi_B \circ E^B )(d)\xi \, , \, \eta \ra_{\CalK_B} .
\end{align*}
\par
Next, following the proof of \cite [Lemma 2.2]{ER:multiplier}, we define non-degenerate
representations $(\pi_A , \CalH_A )$ and $(\pi_C , \CalH_C )$ of $A$ and $C$, which are strongly Morita
equivalent to $(\pi_B , \CalK_B )$ and $(\pi_D , \CalK_D )$, respectively.
\par
We regard $\CalK_B$ as a Hilbert $B-\BC$-bimodule using the representation $(\pi_B , \CalK_B)$
and let $\CalH_A =X\otimes_B \CalK_B$, a Hilbert space, where its inner product $\la \cdot \, , \, \cdot \ra_{\CalH_A}$
is defined by
$$
\la x\otimes\xi \, , \, y\otimes\eta \ra_{\CalH_A} =\la \xi \, , \, \pi_B (\la x, y \ra_B )\eta \ra_{\CalK_B }
$$
for any $x, y\in X$, $\xi , \eta\in \CalK_B$. We define $\pi_A $ by setting
$$
\pi_A (a)(x\otimes\xi )=(a\cdot x)\otimes\xi
$$
for all $a\in A$, $x\in X$, $\xi\in\CalK_B$. Furthermore, we define a linear map $\pi_X$ from $X$
to $\BB(\CalK_B , \CalH_A)$ by setting
$$
\pi_X (x)=x\otimes\xi
$$
for any $x\in X$, $\xi\in\CalK_B$. By \cite [Lemma 2.2]{ER:multiplier}, $(\pi_A , \pi_X , \pi_B )$
is a representation of $X$ on the pair of Hilbert spaces $(\CalH_A , \CalK_B )$. Thus $(\pi_A ,\CalH_A )$ and
$(\pi_B , \CalK_B )$ are strongly Morita equivalent. Similarly we define a representation of $Y$ on the pair of
Hilbert spaces $(\CalH_C \, , \, \CalK_D )$ as follows:
$$
\CalH_C =Y\otimes_D \CalK_D , \quad
\pi_Y (y)\xi =y\otimes\xi , \quad
\pi_C (c)(y\otimes\xi )=(c\cdot y)\otimes\xi 
$$
for any $c\in C$, $y\in Y$, $\xi\in\CalK_D$. 
\par
We consider $\pi_A \circ E^A$, which is a completely positive linear map from $C$ to $\BB(\CalH_A )$.
By \cite [Theorem 4.1]{Paulsen:CB-maps} or \cite [II. 6.9.7]{Blackadar:operator}, there is a minimal
Stinespring representation $(\pi_C' , V_C' , \CalH_C' )$. We will show that $\CalH_C \cong \CalH_C'$ as
Hilbert spaces. In order to do this, we introduce the following Hilbert space $\CalE$. We regard $Y$ as a
Hilbert $C-B$-bimodule in the following way: We define the left $C$-action, the right $B$-action and the left
$C$-valued inner product in the usual way. We define the right $B$-valued inner product by setting
$$
\la x, y\ra_{B} =E^B (\la x, y \ra_D )
$$
for any $x, y\in Y$. We denote it by the symbol $Y_B$. We define the Hilbert space $\CalE$ by
$$
\CalE=Y_B  \otimes_B \CalK_B
$$
in the same way as the definition of Hilbert space $\CalH_A =X\otimes_B \CalK_B$.
\par
First, we show that $\CalH_C$ is isomorphic to $\CalE$ as Hilbert spaces. Before doing it, we prepare
the following lemma.

\begin{lemma}\label{lem:Rel2} With the above notation, let $\{d_{\lambda}\}_{\lambda\in\Lambda}$
be an approximate unit of $D$ with
$0\leq d_{\lambda}$ and $||d_{\lambda}||\leq 1$ for any $\lambda\in\Lambda$.
For any $y\in Y$, $||y-y\cdot d_{\lambda}||\to 0$
$(\lambda\to\infty)$.
\end{lemma}
\begin{proof} Let $y$ be any element in $Y$. For any $\epsilon>0$,
there are $y_1 , y_2 , \dots, y_n \in Y$, $d_1 , d_2 , \dots, d_n \in D$ such that
$$
||y-\sum_{i=1}^n y_i \cdot d_i ||<\epsilon ,
$$
by \cite [Proposition 1.7]{BMS:quasi}.
Then
\begin{align*}
& ||y-y\cdot d_{\lambda}|| \\
& \leq ||y-\sum_{i=1}^n y_i \cdot d_i ||+||\sum_{i=1}^n y_i \cdot d_i -\sum_{i=1}^n y_i \cdot d_i d_{\lambda} ||
+||\sum_{i=1}^n y_i \cdot d_i d_{\lambda}-y\cdot d_{\lambda} || \\
& <2\epsilon+||\sum_{i=1}^n y_i \cdot (d_i -d_i d_{\lambda})|| .
\end{align*}
Since $||d_i -d_i d_{\lambda}||\to 0$ $(\lambda\to\infty)$ for $i=1,2, \dots, n$, there is a $\lambda_0 \in\Lambda$
such that
$$
||\sum_{i=1}^n y_i \cdot (d_i -d_i d_{\lambda})||<\epsilon
$$
for any $\lambda\geq \lambda_0$. Thus $||y-y\cdot d_{\lambda}||\to 0$ $(\lambda\to\infty)$.
\end{proof}

\begin{lemma}\label{lem:Rel3} With the above notation, $\CalH_C \cong \CalE$ as Hilbert spaces.
\end{lemma}
\begin{proof} Let $\Phi$ be the linear map from $Y\odot \CalK_D$ to
$\CalE$ defined by
$$
\Phi(y\otimes (d\otimes \xi))=(y\cdot d )\otimes\xi
$$
for any $y\in Y$, $d\in D$, $\xi\in \CalK_B$, where $Y\odot\CalK_D$ is
the algebraic tensor product of $Y$ and $\CalK_D$. Let $y, y_1 \in Y$, $d, d_1 \in D$, $\xi , \xi_1 \in \CalK_B$.
Then
\begin{align*}
\la y\otimes(d\otimes\xi)\, , \, y_1 \otimes (d_1 \otimes \xi_1 ) \ra_{\CalH_C } & =
\la d\otimes\xi \, , \, \pi_D (\la y \, , \, y_1 \ra_D )d_1\otimes\xi_1 \ra_{\CalK_D } \\
& =\la d\otimes\xi \, , \, \la y \, , \, y_1 \ra_D d_1 \otimes \xi_1 \ra_{\CalK_D} \\
& =\la \xi \, , \, (\pi_B \circ E^B )(d^* \la y \, , \, y_1 \ra_D d_1 )\xi_1 \ra_{\CalK_B} .
\end{align*}
On the other hand,
\begin{align*}
\la \Phi(y\otimes(d\otimes\xi)) \, , \, \Phi(y_1 \otimes(d_1 \otimes\xi_1 ))\ra_{\CalE} & =
\la (y\cdot d)\otimes\xi \, , \, (y_1 \cdot d_1 )\otimes\xi_1 \ra_{\CalE} \\
& =\la \xi , \, \pi_B (\la y\cdot d \, , \, y_1 \cdot d_1 \ra_B )\xi_1 \ra_{\CalK_B} \\
& =\la \xi \, , \, (\pi_B \circ E^B )(\la y\cdot d \, , \, y_1 \cdot d_1 \ra_D )\xi_1 \ra_{\CalK_B} \\
& =\la \xi \, , \, (\pi_B \circ E^B )(d^* \la y \, , \, y_1 \ra_D d_1 )\xi_1 \ra_{\CalK_B} .
\end{align*}
Thus $\Phi$ preserves the inner products on the algebraic tensor products. We can extend $\Phi$ to
$\CalH_C$. We denote it by the same symbol $\Phi$. Then $\Phi$ is an
isometry from $\CalH_C$ to $\CalE$. Next, we show that $\Phi$ is surjective.
Let $y$ and $\xi$ be elements in $Y_B$ and $\CalK_B$, respectively.
Let $\{d_{\lambda}\}_{\lambda\in\Lambda}$ be an approximate units of $D$ with
$d_{\lambda}\geq 0$ and $||d_{\lambda}||\leq 1$.
Then by Lemma \ref {lem:Rel2}, $y=\lim_{\lambda}y\cdot d_{\lambda}$. Also,
$y\otimes (d_{\lambda}\otimes\xi)$ is an element in $Y\otimes_D \CalK_D$ and
$$
\Phi(y\otimes(d_{\lambda}\otimes\xi))=y\cdot d_{\lambda}\otimes\xi \to y\otimes\xi \quad
(\lambda\to\infty) .
$$
Since $\Phi$ is isometric, $\{y\otimes(d_{\lambda}\otimes\xi )\}_{\lambda\in\Lambda}$ is a Cauchy
net in $Y\otimes_D \CalK_D$. Hence there exists an element $z\in \CalH_C$ such that
$y\otimes(d_{\lambda}\otimes\xi)\to z$
$(\lambda\to\infty)$. Thus $\Phi(z)=y\otimes \xi$, that is, $\Phi$ is surjective. Therefore $\Phi$ is an isometry
from $\CalH_C$ onto $\CalE$.
\end{proof}

By the proof of Lemma \ref {lem:Rel3}
$$
\Phi^* (y\otimes\xi)=\lim_{\lambda}y\otimes(d_{\lambda}\otimes\xi)
$$
for any $y\in Y$, $\xi\in \CalK_B$, where $\{d_{\lambda}\}_{\lambda\in\Lambda}$ is an
approximate units of $D$ with $0\leq d_{\lambda}$ and $||d_{\lambda}||\leq 1$ for any $\lambda\in\Lambda$.
\par
Next, we will show that $\CalH_C ' \cong \CalE$ as Hilbert spaces.

\begin{lemma}\label{lem:Rel1} With the above notation $\CalH_C' \cong \CalE$ as Hilbert spaces.
\end{lemma}
\begin{proof} Let $c\in C$, $x\in X$, $\xi\in \CalK_B$. We denote by the same notation
$c\otimes (x\otimes\xi)$, the equivalence class of $c\otimes(x\otimes\xi)\in C\odot\CalH_A$,
where $C\odot \CalH_A$ is the algebraic tensor product of $C$ and $\CalH_A$.
Let $\Psi$ be the linear map from $C\odot\CalH_A$ to $\CalE$
defined by
$$
\Psi(c\otimes(x\otimes\xi))=(c\cdot x)\otimes\xi
$$
for any $c\in C$, $x\in X$, $\xi\in\CalK_B$, where we note that $X$ is a closed subspace of $Y$.
Let $c, c_1 \in C$, $x, x_1 \in X$, $\xi, \xi_1 \in \CalK_B$. Then
\begin{align*}
\la c\otimes(x\otimes\xi) \, , \, c_1 \otimes(x_1 \otimes\xi_1 ) \ra_{\CalH_C'} & =
\la x\otimes\xi \, , \, (\pi_A \circ E^A )(c^* c_1 )(x_1 \otimes\xi_1 ) \ra_{\CalH_A} \\
& =\la x\otimes\xi \, , \, [E^A (c^* c_1 )\cdot x_1 ]\otimes\xi_1 \ra_{\CalH_A} \\
& =\la \xi \, , \, \pi_B (\la x \, , \, E^A (c^* c_1 )\cdot x_1 \ra_B )\, \xi_1 \ra_{\CalK_B} .
\end{align*}
Since $E^A$ and $E^B$ are strongly Morita equivalent with respect to $Y$ and its closed subspace
$X$ as bimodule linear maps, we have the equation
$$
\la z \, , \, E^A (a)\cdot z_1 \ra_B =E^B(\la z \, , \, a\cdot z_1 \ra_D )
$$
for any $z, z_1 \in X$, $a\in C$. Hence we obtain that
$$
\la c\otimes(x\otimes\xi) \, , \, c_1 \otimes(x_1 \otimes\xi_1 )\ra_{\CalH_C'}
=\la \xi \, , \, (\pi_B \circ E^B)(\la x \, , \, c^* c_1 \cdot x_1 \ra_D)\xi_1 \ra_{\CalK_B} .
$$
On the other hand,
\begin{align*}
\la \Psi(c\otimes(x\otimes\xi)) \, , \, \Psi(c_1 \otimes(x_1 \otimes\xi_1 ))\ra_{\CalE} & =
\la (c\cdot x)\otimes\xi \, , \, (c_1 \cdot x_1 )\otimes\xi_1 \ra_{\CalE} \\
& =\la \xi \, , \, (\pi_B \circ E^B )(\la c\cdot x \ , \, c_1 \cdot x_1 \ra_D )\, \xi_1 \ra_{\CalK_B} .
\end{align*}
Thus $\Psi$ preserves the inner products on the algebraic tensor products. We can extend $\Psi$
to $\CalH_C '$. We denote it by the same symbol $\Psi$.
Then $\Psi$ is an isometry from $\CalH_C '$ to $\CalE$. We show that $\Psi$ is
surjective. Let $y$ and $\xi$ be any elements in $Y_B$ and $\CalK_B$, respectively.
Brown, Mingo and Shen \cite [Proposition 1.7]{BMS:quasi}, $Y=\overline{Y\cdot D}$. Hence
\cite [Definition 2.1]{KT4:morita}
$$
Y=\overline{Y\cdot D}=\overline{Y\cdot \la Y, X\ra_D }=\overline{{}_C \la Y, Y \ra\cdot X}=\overline{C\cdot X} .
$$
For any $m\in\BN$, there are elements $c_1 , c_2 , \dots, c_{n_m}\in C$ and $x_1 , x_2 \dots, x_{n_m}\in X$ such that
$$
||y-\sum_{i=1}^{n_m} c_i \cdot x_i ||<\frac{1}{m} .
$$
Let $z_m =\sum_{i=1}^{n_m}c_i \otimes(x_i \otimes\xi)$. Then $z_m\in \CalH_C '$
for any $m\in\BN$ and $\Psi(z_m )=\sum_{i=1}^{n_m }(c_i \cdot x_i )\otimes \xi\to y\otimes \xi$ $(m\to\infty)$.
Since $\Psi$ is isometric, $\{z_m \}$ is a Cauchy sequence in $\CalH_C '$. Thus there exists
an element $z\in \CalH_C '$ such that $z_m \to z$ $(m\to\infty)$.
Hence $\Psi(z)=y$, that is, $\Psi$ is surjective. Therefore, $\Psi$ is an isometry from $\CalH_C '$
onto $\CalE$.
\end{proof}

\begin{lemma}\label{lem:Rel4} With the above notation, the non-degenerate representations
$(\pi_C , \CalH_C )$ and $(\pi_C ' \, , \, \CalH_C ')$ are unitarily equivalent.
\end{lemma}
\begin{proof} Let $c\in C$, $x\in X$ and $\xi\in\CalK_B$. Then
$$
(\Phi^* \Psi)(c\otimes (x\otimes \xi))=\Phi^* ((c\cdot x)\otimes\xi)
=\lim_{\lambda}(c\cdot x)\otimes (d_{\lambda}\otimes\xi) ,
$$
where $\{d_{\lambda}\}_{\lambda\in\Lambda}$ is an approximate unit of $D$ with
$d_{\lambda}\geq 0$ and $||d_{\lambda}||\leq 1$ for any $\lambda\in\Lambda$.
Thus for any $c_1 \in C$,
$$
(\Phi^* \Psi\pi_C '(c_1 ))(c\otimes(x\otimes\xi))=\lim_{\lambda}(c_1 c\cdot x)\otimes(d_{\lambda}\otimes\xi) .
$$
On the other hand,
$$
(\pi_C (c_1 )\Phi^* \Psi)(c\otimes(x\otimes\xi))=\lim_{\lambda}\pi_C (c_1 )((c\cdot x)\otimes(d_{\lambda}\otimes\xi))
=\lim_{\lambda}(c_1 c \cdot x)\otimes (d_{\lambda}\otimes\xi) .
$$
Hence $\pi_C (c_1 )\Phi^* \Psi=\Phi^* \Psi \pi_C ' (c_1 )$ for any $c_1 \in C$.
Therefore, we obtain the conclusion.
\end{proof}

Let $U=\Phi^* \Psi$. Then $U$ is an isometry from $\CalH_C '$ onto $\CalH_C$ and $\pi_C (c)=U\pi_C ' (c)U^*$
for any $c\in C$.

\begin{lemma}\label{lem:Rel5} With the above notation, $(\pi_C \, , \, UV_C ' \, , \, \CalH_C )$ is a minimal
Stinespring representation for $\pi_A \circ E^A$.
\end{lemma}
\begin{proof} For any $c\in C$,
$$
(UV_C ' )^* \pi_C (c)UV_C ' =V_C ^{' * } U^* \pi_C (c)UV_C ' 
=V_C^{' *}\pi_C ' (c)V_C ' =(\pi_A \circ E^A )(c) .
$$
Also,
$$
\overline{\pi_C (C)UV_C ' \CalH_A }=U\overline{\pi_C ' (C)V_C ' \CalH_A } =U\CalH_C ' =\CalH_C .
$$
Hence $(\pi_C \, , \, UV_C ' \, , \, \CalH_C )$ is a minimal
Stinespring representation for $\pi_A \circ E^A$.
\end{proof}

\begin{lemma}\label{lem:Rel6} With the above notation, $\pi_A \circ E^A$ and $\pi_B \circ E^B$ are
strongly Morita equivalent as completely positive linear maps.
\end{lemma}
\begin{proof} By the definition of $(\pi_C , \CalH_C )$, $(\pi_C , \CalH_C )$ and $(\pi_D , \CalK_D )$ are
strongly Morita equivalent. Also, since $(\pi_C \, , \, UV_C ' \, , \, \CalH_C )$ is a minimal
Stinespring representation for $\pi_A \circ E^A$ by Lemma \ref {lem:Rel5}, $\pi_A \circ E^A$
and $\pi_B \circ E^B$ are strongly Morita equivalent as completely positive linear maps.
\end{proof}

Combining the above discussions, we obtain the following theorem.

\begin{thm}\label{thm:Rel7} Let $A\subset C$ and $B\subset D$ be inclusions of $C^*$-algebras
with $\overline{AC}=C$ and $\overline{BD}=D$. Let $E^A$ and $E^B$ conditional expectations
from $C$ and $D$ onto $A$ and $B$, respectively. We assume that $E^A$ and $E^B$ are
strongly Morita equivalent with respect to a $C-D$-equivalence bimodule $Y$ and its closed subspace $X$.
Then for any non-degenerate representation $(\pi_B , \CalK_B )$ of $B$, there exists a non-degenerate
representation $(\pi_A , \CalH_A )$ of $A$ such that $\pi_A \circ E^A$ and $\pi_B \circ E^B$ are
strongly Morita equivalent as completely positive linear maps.
\end{thm}
\begin{proof} This is immediate by Lemma \ref {lem:Rel6}.
\end{proof}

Next, we will consider the inverse direction. Let $A\subset C$ and $B\subset D$ be as above.
We suppose that $A\subset C$ and $B\subset D$ are strongly Morita equivalent
with respect to a $C-D$-equivalence bimodule $Y$ and its closed subspace $X$.
Let $E^A$ and $E^B$ be conditional expectations from $C$ and $D$ onto $A$ and $B$, respectively.
Let $(\pi_B , \CalK_B )$ be a non-degenerate
representation $B$ and $(\pi_A , \CalH_A )$ be the non-degenerate representation of $A$
induced by $X$ and $(\pi_B , \CalK_B )$, which is defined in \cite [Lemma 2,2]{ER:multiplier}.
Let $(\pi_D, V_D , \CalK_D )$ be a minimal Stinespring representation for $\pi_B \circ E^B$ and
$(\pi_C , \CalH_C )$ the non-degenerate representation of $C$ induced by $Y$ and
$(\pi_D , \CalK_D )$. First, we show the following lemma.

\begin{lemma}\label{lem:Rel8} Let $\{d_{\lambda}\}_{\lambda\in\Lambda}$ be an approximate
unit of $D$ with $d_{\lambda}\ge 0$ and $||d_{\lambda}||\le 1$ for any $\lambda\in\Lambda$.
Then $\{y\otimes (d_{\lambda}\otimes\xi)\}_{\lambda\in\Lambda}$ is a Cauchy net in $\CalH_C$
with respect to the weak topology of $\CalH_C$ for any $y\in Y$, $\xi\in \CalK_B$.
\end{lemma}
\begin{proof} Since the linear span of the set
$$
\{y\otimes (d\otimes\xi)\in \CalH_C \, | \, y\in Y, \, d\in D, \, \xi\in \CalK_B \}
$$
is dense in $\CalH_C$,  it suffices to show that for any $y_1 \in Y$, $\xi_1 \in \CalK_B$, $d_1 \in D$,
the net $\{ \la y\otimes(d_{\lambda}\otimes \xi) \, , \, y_1 \otimes(d_1 \otimes \xi_1 )\ra_{\CalH_C} \}_{\lambda}$
is a Cauchy net. For any $\lambda, \mu\in\Lambda$,
\begin{align*}
& \la y\otimes((d_{\lambda}-d_{\mu})\otimes\xi) \, , \, y_1 \otimes(d_1 \otimes\xi_1 )\ra_{\CalH_C} \\
& =\la (d_{\lambda}-d_{\mu})\otimes\xi \, , \, \pi_D (\la y\, ,\, y_1 \ra_D )(d_1 \otimes \xi_1 ) \ra_{\CalK_D} \\
& =\la (d_{\lambda}-d_{\mu})\otimes\xi \, , \, \la y \, , \, y_1 \ra_D \, d_1 \otimes\xi_1 \ra_{\CalK_D} \\
& =\la \xi \, , \, (\pi_B \circ E^B )((d_{\lambda}-d_{\mu})\la y\, , \, y_1 \ra_D \, d_1 )\xi_1 \ra_{\CalK_B}\to 0 \quad
(\lambda\, , \, \mu\to\infty) .
\end{align*}
Thus, we obtain the conclusion.
\end{proof}

Let $V_C$ be the linear map from $\CalH_A$ to $\CalH_C$ defined by
$$
V_C (x\otimes \xi )=\lim_{\lambda}x\otimes (d_{\lambda}\otimes\xi)
$$
for any $x\in X$, $\xi\in\CalK_B$, where $\{d_{\lambda}\}_{\lambda\in\Lambda}$ is an approximate
unit of $D$ with $d_{\lambda}\ge 0$ and $||d_{\lambda}||\leq 1$ for any $\lambda\in\Lambda$ and
the limit is taken under the weak topology of $\CalH_C$. By Lemma \ref {lem:Rel8}, the above limit
is convergent with respect to the weak topology of $\CalH_C$ and by routine computations, $V_C$ is
well-defined and independent of the choice of an approximate unit of $D$.

\begin{lemma}\label{lem:Rel9} With the above notation, $V_C$ is an isometry from $\CalH_A$ to $\CalH_C$.
\end{lemma}
\begin{proof} For any $x, x_1 \in X$, $\xi, \xi_1 \in \CalK_B$,
\begin{align*}
\la V_C(x\otimes\xi) \, , \, V_C (x_1 \otimes\xi_1 )\ra_{\CalH_C} & =
\lim_{\lambda, \, \mu}\la x\otimes (d_{\lambda}\otimes\xi) \, , \, x_1 \otimes (d_{\mu}\otimes\xi_1 ) \ra_{\CalH_C} \\
& =\lim_{\lambda , \, \mu}\la d_{\lambda}\otimes\xi \, , \, \pi_D (\la x, x_1 \ra_B )(d_{\mu}\otimes\xi_1 )\ra_{\CalK_D} \\
& =\lim_{\lambda , \, \mu}\la d_{\lambda}\otimes\xi \, , \, \la x, x_1 \ra_B \, d_{\mu}\otimes\xi_1 \ra_{\CalK_D} \\
& =\lim_{\lambda , \, \mu}\la \xi \, , \, (\pi_B \circ E^B )(d_{\lambda} \la x, \, x_1 \ra_B \, d_{\mu})\xi_1
\ra_{\CalK_B} \\
& =\la \xi \, , \, \pi_B (\la x, \, x_1 \ra_B )\xi_1 \ra_{\CalK_B} \\
& =\la x\otimes\xi \, , \, x_1 \otimes\xi_1 \ra_{\CalH_A} .
\end{align*}
Thus, we obtain the conclusion.
\end{proof}

\begin{prop}\label{prop:Rel10} Let $V_C$ be the isometry from $\CalH_A$ to $\CalH_C$ defined by
$$
V_C (x\otimes \xi)=\lim_{\lambda}x \otimes (d_{\lambda}\otimes\xi)
$$
for any $x\in X$, $\xi\in\CalK_B$, where $\{d_{\lambda}\}_{\lambda\in\Lambda}$ is an approximate unit
of $D$ with $d_{\lambda}\ge 0$ and $||d_{\lambda}||\le 1$ for any $\lambda\in\Lambda$ and the limit
is taken under the weak topology of $\CalH_C$. We suppose that $(\pi_B , \CalK_B )$ is faithful and that
$$
(\pi_A \circ E^A )(c)=V_C^* \pi_C (c)V_C
$$
for any $c\in C$. Then $E^A$ and $E^B$ are strongly Morita equivalent as bimodule linear maps.
\end{prop}
\begin{proof} For any $c\in C$, $x, x_1 \in X$ and $\xi, \xi_1 \in \CalK_B$,
\begin{align*}
\la (\pi_A \circ E^A )(c)(x\otimes\xi) \, , \, x_1 \otimes\xi_1 \ra_{\CalH_A} & =
\la (E^A (c)\cdot x)\otimes \xi \, , \, x_1 \otimes \xi_1 \ra_{\CalH_A} \\
& =\la \xi \, , \, \pi_B (\la (E^A (c)\cdot x)\, , \, x_1 \ra_B )\xi_1 \ra_{\CalK_B} .
\end{align*}
Also,
\begin{align*}
\la V_C^* \pi_C(c)V_C (x\otimes \xi)\, , \, x_1 \otimes\xi_1 \ra_{\CalH_A} & =
\la \pi_C (c)V_C (x\otimes\xi) \, , \, V_C (x_1 \otimes\xi_1 )\ra_{\CalH_C} \\
& =\lim_{\lambda, \, \mu}\la \pi_C (c)(x\otimes(d_{\lambda}\otimes\xi))\, , \, x_1 \otimes
(d_{\mu}\otimes\xi_1 )\ra_{\CalH_C} \\
& =\lim_{\lambda, \, \mu}\la (c\cdot x)\otimes(d_{\lambda}\otimes\xi)\, , \, x_1 \otimes
(d_{\mu}\otimes\xi_1 )\ra_{\CalH_C} \\
& =\lim_{\lambda, \, \mu}\la d_{\lambda}\otimes\xi\, , \, \pi_D (\la c\cdot x\, , \, x_1 \ra_D )
d_{\mu}\otimes\xi_1\ra_{\CalK_D} \\
& =\lim_{\lambda}\la d_{\lambda}\otimes\xi \, , \, \la c\cdot x \, , \, x_1 \ra_D \otimes\xi_1 \ra_{\CalK_B} \\
& =\lim_{\lambda}\la \xi \, , \, (\pi_B \circ E^B )(d_{\lambda} \, \la c\cdot x \, , \, x_1 \ra_D )\xi_1 \ra_{\CalK_B} \\
& =\la \xi \, , \, (\pi_B \circ E^B )(\la c\cdot x \, , \, x_1 \ra_D )\xi_1 \ra_{\CalK_B} .
\end{align*}
Since $(\pi_A \circ E^A)(c)=V_C^* \pi_C (c)V_C$ for any $c\in C$ and $\pi_B$ is faithful,
$$
E^B (\la c\cdot x \, , \, x_1 \ra_D )=\la (E^A (c)\cdot x) \, , \, x_1 \ra _B .
$$
for any $c\in C$, $x, x_1 \in X$. By \cite [Lemma 2.5]{Kodaka:Picard3}, $E^A$ and $E^B$ are strongly Morita
equivalent as bimodule linear maps.
\end{proof}

\section{A correspondence of strong Morita equivalence classes of completely positive linear maps}
\label{sec:Co} Let $A$ and $B$ be $C^*$-algebras, which are strongly Morita equivalent with
respect to an $A-B$-equivalence bimodule $X$. In this section, we will construct a $1-1$ correspondence between
the set of all strong Morita equivalence classes of completely positive linear maps on $A$ and the set of
all strong Morita equivalence classes of completely positive linear maps on $B$ and we will show that
the corresponding positive linear maps are strongly Morita equivalent.
\par
Let $\psi$ be a completely positive linear map from $B$ to $\BB(\CalK)$, where $\CalK$ is a Hilbert
space. Let $(\pi_{\psi}, V_{\psi}, \CalK_{\psi})$ be a minimal Stinespring representation for $\psi$.
Let $(\pi_A , \CalH_A )$ be the non-degenerate representation of $A$ induced by $X$ and
$(\pi_{\psi}, V_{\psi}, \CalK_{\psi})$. Let $\{u_i \}_{i=1}^n$ be a finite subset of $X$.
Let $\phi$ be the linear map from $A$ to $\BB(\CalK)\otimes M_n (\BC)$ defined by
$$
\left[\phi(a)_{ij}\right]_{i,j=1}^n =\left[\psi(\la u_i \, , \, a\cdot u_j \ra_B )\right]_{i, j=1}^n
$$
for any $a\in A$. Since $\psi(b)=V_{\psi}^* \pi_{\psi}(b)V_{\psi}$ for any $b\in B$,
\begin{align*}
\left[ \phi(a)_{i j}\right]_{i, j=1}^n & =\left[V_{\psi}^* \pi_{\psi}
(\la u_i \, , \, a\cdot u_j \ra_B )V_{\psi} \right]_{i, j=1}^n \\
& =(V_{\psi}^* \otimes I_n )\left[\pi_{\psi}(\la u_i \, , \, a\cdot u_j \ra_B )\right]_{i, j=1}^n (V_{\psi}\otimes I_n ) .
\end{align*}
\par
First, we will show that $\phi$ is a completely positive linear map from $A$ to $\BB(\CalK)\otimes M_n (\BC)$.
\par
Let $X^n$ be the $n$-times direct sum of $X$ and we regard $X^n$ as an $A-M_n (B)$-equivalence bimodule
as follows: For any $a\in A$, $\left[ b_{ij}\right]_{i, j=1}^n \in M_n (B)$, $\left[ x_1 , \dots, x_n \right]$,
\newline
$\left[ z_1 , \dots , z_n \right]\in X^n$, we define the left $A$-action, right $M_n (B)$-action and the left
$A$-valued inner product, the right $M_n (B)$-valued inner product on $X^n$ by setting
\begin{align*}
a\cdot \left[ x_1 , \dots, x_n \right] & =\left[ a\cdot x_1 , \dots, a\cdot x_n \right] , \\
\left[ x_1 , \dots, x_n \right]\cdot \left[ b_{ij}\right]_{i, j=1}^n & =\left[\sum_{i=1}^n x_i \cdot b_{i1}, \dots ,
\sum_{i=1}^n x_i \cdot b_{i\, n}\right], \\
{}_A \la \left[ x_1 , \dots, x_n \right]\, , \, \left[ z_1 , \dots, z_n \right] \ra & =
\sum_{i=1}^n {}_A \la x_i \, , \, z_i \ra , \\
\la \left[ x_1 , \dots, x_n \right]\, , \, \left[ z_1 , \dots, z_n \right] \ra_{M_n (B)} & =
\left[ \la x_i \, , \, z_j \ra_B \right]_{i, j=1}^n .
\end{align*}
Let $Y=X^n$ and $D=M_n (B)$. We regard $Y$ as an $A-D$-equivalence bimodule in the abovev way.
For each $m\in \BN$, let $M_m (Y)$ be the $\BC$-linear space of all matrices over $Y$. We regard $M_m (Y)$
as an $M_m (A)-M_m (D)$-equivalence bimodule as follows: For any $\left[a_{kl}\right]_{k, l=1}^m \in M_m (A)$,
$\left[d_{kl}\right]_{k, l=1}^m \in M_m (D)$, $\left[y_{kl} \right]_{k, l=1}^m$, $\left[z_{kl}\right]_{k, l=1}^m M_ m(Y)$,
we define the left $M_m (A)$-action, the right $M_m (D)$-action and the left $M_m (A)$-valued inner product,
the right $M_m (D)$-valued inner product on $M_m (Y)$ by setting
\begin{align*}
\left[a_{kl}\right]_{k, l=1}^m \cdot\left[ y_{kl}\right]_{k, l=1}^m & =
\left[\sum_{t=1}^m a_{kt}\cdot b_{tl}\right]_{k, l=1}^m , \\
\left[ y_{kl}\right]_{k, l=1}^m \cdot \left[ d_{kl}\right]_{k, l=1}^m & =
\left[ \sum_{t=1}^m y_{kt}\cdot d_{tl} \right]_{k, l=1}^m ,  \\
{}_{M_m (A)} \la \left[ y_{kl}\right]_{k, l=1}^m \, ,\, \left[ z_{kl}\right]_{k, l=1}^m \ra & =
\left[ \sum_{t=1}^m {}_A \la y_{kt} \, , \, z_{lt} \ra \right]_{k, l=1}^m , \\
\la \left[ y_{kl}\right]_{k, l=1}^m \, ,\, \left[ z_{kl}\right]_{k, l=1}^m \ra_{M_m (D)} & =
\left[ \sum_{t=1}^m \la y_{tk} \, , \, z_{tl} \ra_D \right]_{k, l=1}^m .
\end{align*}

\begin{lemma}\label{lem:Co1} With the above notation, $\phi$ is a completely positive linear map
from $A$ to $\BB(\CalK)\otimes M_n (\BC)$.
\end{lemma}
\begin{proof} Since $\phi$ is clearly linear, we have only to show that for any $\left[a_{kl}\right]_{k,l=1}^m \in M_m (A)$
with $\left[a_{kl}\right]_{k,l=1}^m \geq 0$,
$$
\left[\phi(a_{kl})\right]_{k, l=1}^m \geq 0 .
$$
Let $\left[a_{kl}\right]_{k,l=1}^m$ be any positive element in $M_m (A)$. Then by the definition of $\phi$,
$$
\phi(a_{kl})_{i j}=\psi(\la u_i \, , \, a_{kl}\cdot u_j \ra_B ) .
$$
Thus
\begin{align*}
& (\phi\otimes\id_{M_m(\BC)})(\left[a_{kl}\right]_{k,l=1}^m ) \\
& =\begin{bmatrix}\left[ \psi(\la u_i \, , \, a_{11}\cdot u_j \ra_B ) \right]_{i, j=1}^n  & \ldots 
& \left[ \psi(\la u_i \, , \, a_{1m}\cdot u_j \ra_B ) \right]_{i, j=1}^n \\
\vdots & \ddots & \vdots \\
\left[ \psi(\la u_i \, , \, a_{m1}\cdot u_j \ra_B ) \right]_{i, j=1}^n & \ldots
& \left[ \psi(\la u_i \, , \, a_{m \, m}\cdot u_j \ra_B ) \right]_{i, j=1}^n
\end{bmatrix}_{k, l=1}^m \\
& =\left[ \left[ \psi(\la u_i \, , \, a_{kl}\cdot u_j \ra_B ) \right]_{i, j=1}^n \right]_{k, l=1}^m .
\end{align*}
Since $\psi$ is a completely positive linear map from $B$ to $\BB(\CalK)$,
we have only to show that the element
$$
\left[ \left[ \la u_i \, , \, a_{kl}\cdot u_j \ra_B \right]_{i, j=1}^n \right]_{k, l=1}^m 
$$
is positive in $B\otimes M_n (\BC)\otimes M_m(\BC)$. Let $y=\left[ u_1 , \dots , u_n \right]\in Y$.
Let $\left[ y_{kl}\right]_{k. l=1}^m $ be an element in $M_m (Y)$ defined by
$$
y_{kl} = \begin{cases} y & \text{if $k=l$} \\
0 & \text{if $k\ne l$} \end{cases} .
$$
Then
$$
\la \left[ y_{kl}\right]_{k, l=1}^m \, , \, \left[ a_{kl}\right]_{k, l=1}^m \cdot \left[ y_{kl}\right]_{k, l=1}^m
\ra_{M_m(D)} \geq 0
$$
since $\left[ a_{kl}\right]_{k, l=1}^m \geq 0$. On the other hand, by the definition of $\left[ y_{kl}\right]_{k, l=1}^m$
and $y$,
\begin{align*}
\la \left[ y_{kl}\right]_{k, l=1}^m \, , \, \left[ a_{kl}\right]_{k, l=1}^m & \cdot \left[ y_{kl}\right]_{k, l=1}^m
\ra_{M_m(D)} =\la \left[ y_{kl}\right]_{k, l=1}^m \, , \, \left[ \sum_{t=1}^m a_{kt}\cdot y_{tl}\right]_{k, l=1}^m
\ra_{M_m(D)} \\
& =\la \left[ y_{kl}\right]_{k, l=1}^m \, , \, \left[ a_{kl}\cdot y \right]_{k, l=1}^m
\ra_{M_m(D)} \\
& =\left[ \sum_{t=1}^m \la y_{tk} \, , \, a_{tl}\cdot y \ra_D \right]_{k, l=1}^m \\
& =\left[ \la y \, , \, a_{kl}\cdot y \ra_D \right]_{k, l=1}^m \\
& =\left[ \la \left[u_1 , \dots , u_n \right] \, , \, 
\left[a_{kl}\cdot u_1 , \dots , a_{kl}\cdot u_n \right] \ra_{M_n (B)}\right]_{k. l=1}^m \\
& =\left[ \left[ \la u_i \, , \, a_{kl}\cdot u_j \ra_B  \right]_{i, j=1}^n \right]_{k, l=1}^m .
\end{align*}
Therefore, we obtain the conclusion.
\end{proof}

Let $(\pi_{\phi}, V_{\phi}, \CalH_{\phi})$ be a minimal Stinespring representation for $\phi$.
Let $A\odot (\CalK\otimes\BC^n )$ be the algebraic tensor product of $A$ and $\CalK\otimes\BC^n$.
Let $\{b_p \}_{p\in P}$ be an approximate unit of $B$ with $b_p \geq 0$ and $||b_p ||\leq 1$ for any
$p\in P$.
We define a map $U$ from $A\odot (\CalK\otimes\BC^n )$ to $\CalH_A$ by setting
$$
U(a\otimes\xi\otimes\lambda)=\lim_p \sum_{i=1}^n \lambda_i (a\cdot u_i )\otimes b_p \otimes\xi
$$
for any $a\in A$, $\xi\in \CalK$, $\lambda\in\BC^n$ and $\lambda=\left[ \begin{array}{c}
\lambda_1 \\
\vdots \\
\lambda_n \end{array}\right]$,
and extending linearly, where we identify $\BB(\CalK)\otimes M_n (\BC)$ with
$\BB(\CalK\otimes\BC^n )$ and the limit is taken under the weak topology of $\CalH_{A}$.

\begin{lemma}\label{lem:Co2} With the above notation, $U$ is an isometry from
$A\odot(\CalK\otimes\BC^n )$ to $\CalH_A $. Hence we can extend $U$ to an isometry
from $\CalH_{\phi}$ to $\CalH_A$.
\end{lemma}
\begin{proof} Let $a, b\in A$, $\xi, \eta\in\CalK$, $\lambda=\left[ \begin{array}{c}
\lambda_1 \\
\vdots \\
\lambda_n \end{array}\right] \, , \, \mu=\left[ \begin{array}{c}
\mu_1 \\
\vdots \\
\mu_n \end{array}\right]\in \BC^n$. Then
\begin{align*}
& \la U(a\otimes\xi\otimes\lambda) \, , \, U(b\otimes\xi\otimes\mu) \ra_{\CalH_A} \\
& =\la \lim_p \sum_{i=1}^n \lambda_i (a\cdot u_i )\otimes b_p \otimes\xi \quad , \quad
\lim_q \sum_{j=1}^n \mu_j (b\cdot u_j )\otimes b_q \otimes\eta \ra_{\CalH_A} \\
& =\lim_{p, q}\sum_{i, j=1}^n \la b_p \otimes\xi \, , \, \pi_{\psi}(\la \lambda_i (a\cdot u_i ) \, , \,
\mu_j (b\cdot u_j )\, \ra_B \, )(b_q \otimes\eta) \, \ra_{\CalK_{\psi}} \\
& =\sum_{i, j=1}^n \la V_{\psi}\xi \, , \, \pi_{\psi}(\la \lambda_i (a\cdot u_i ) \, , \,
\mu_j (b\cdot u_j )\, \ra_B \, )V_{\psi}\eta \, \ra_{\CalK_{\psi}} \\
& =\sum_{i, j=1}^n \la \xi \, , \, V_{\psi}^* \pi_{\psi}(\la \lambda_i (a\cdot u_i ) \, , \,
\mu_j (b\cdot u_j )\, \ra_B \, )V_{\psi}\eta \, \ra_{\CalK} \\
& =\sum_{i, j=1}^n \la \xi \, , \, \psi(\la \lambda_i (a\cdot u_i ) \, , \,
\mu_j (b\cdot u_j )\, \ra_B \, )\eta \, \ra_{\CalK} .
\end{align*}
On the other hand,
\begin{align*}
\la a\otimes\xi\otimes\lambda \, , \, b\otimes\eta\otimes\mu \ra_{\CalH_{\phi}}
& =\la \left[ \begin{array}{c}
\lambda_1 \xi\\
\vdots \\
\lambda_n \xi \end{array}\right]  \, , \,
\left[\phi(a^* b)_{ij}\right]_{i, j=1}^n
\left[ \begin{array}{c}
\mu_1 \eta\\
\vdots \\
\mu_n \eta \end{array}\right] \, \ra_{\CalK\otimes\BC^n } \\
& =\la \left[ \begin{array}{c}
\lambda_1 \xi\\
\vdots \\
\lambda_n \xi \end{array}\right]  \, , \,
\left[ \begin{array}{c}
\sum_{j=1}^n \phi(a^* b)_{1j}\mu_j \eta \\
\vdots \\
\sum_{j=1}^n \phi(a^* b)_{nj}\mu_j \eta \end{array}\right] \, \ra_{\CalK\otimes\BC^n } \\
& =\sum_{i, j=1}^n \la \lambda_i \xi \, , \, \phi(a^* b)_{ij}\mu_j \eta \, \ra_{\CalK} \\
& =\sum_{i, j=1}^n \la \lambda_i \xi \, , \, \psi(\la u_i \, , \, a^* b \cdot u_j \ra_B )\mu_j \eta \, \ra_{\CalK} \\
& =\sum_{i, j=1}^n \la \xi \, , \, \psi(\la \lambda_i (a\cdot u_i ) \, , \, \mu_j (b\cdot u_j ) \ra_B )\, \eta \, \ra_{\CalK} .
\end{align*}
Hence $U$ preserves the inner products on the algebraic tensor products. We can extend $U$ to an isometry
from $\CalH_{\phi}$ to $\CalH_A$.
\end{proof}

We denote by the same symbol the above isometry from $\CalH_{\phi}$ to $\CalH_A$.
From now on, we assume that $A$ and $B$ are strongly Morita equivalent unital $C^*$-algebras.
Since $X$ is an $A-B$-equivalence bimodule, by Kajiwara and Watatani \cite [Corollary 1.19]{KW1:bimodule},
there is a left $A$-basis in $X$, which is a finite subset of $X$. We will show that $\phi$ is strongly Morita
equivalent to $\psi$ if $\{u_i \}_{i=1}^n $ is a left $A$-basis in $X$.

\begin{lemma}\label{lem:Co3} With the above notation, we assume that $\{u_i \}_{i=1}^n $ is a left $A$-basis
in $X$. Then $U$ is surjective.
\end{lemma}
\begin{proof} Since $\{u_i \}_{i=1}^n $ is a left $A$-basis in $X$,
for any $x\in X$, $x=\sum_{i=1}^n {}_A \la x, u_i \ra \cdot u_i$. Thus the set
$\{\sum_{i=1}^n a\cdot u_i \, | \, a\in A \}$ is equal to $X$. Also, by \cite [Proposition 1.7]{BMS:quasi},
$\overline{X\cdot B}=X$. Hence we can see that the set
$$
\{\sum_{i=1}^n (a\cdot u_i )\otimes 1_B \otimes \xi \, | \, a\in A \, , \, \xi\in \CalK \}
$$
is dense in $\CalH_A$. Therefore, $U$ is surjective.
\end{proof}

\begin{lemma}\label{lem:Co4} With the above notation, we assume that $\{u_i \}_{i=1}^n $ is a left $A$-basis
in $X$. Then
$$
\pi_{\phi}(c)=U^* \pi_A (c)U
$$
for any $c\in A$.
\end{lemma}
\begin{proof} Let $a, b, c\in A$, $\xi, \eta\in\CalK$, $\lambda=\left[ \begin{array}{c}
\lambda_1 \\
\vdots \\
\lambda_n \end{array}\right] \, , \, \mu=\left[ \begin{array}{c}
\mu_1 \\
\vdots \\
\mu_n \end{array}\right]\in \BC^n$. Then
\begin{align*}
& \la U^* \pi_A (c)U(a\otimes\xi\otimes\lambda) \, , \, b\otimes\eta\otimes\mu \ra_{\CalH_{\phi}} \\
& =\la \pi_A (c)\sum_{i=1}^n \lambda_i (a\cdot u_i )\otimes 1_B \otimes\xi \, , \,
\sum_{j=1}^n \mu_j (b\cdot u_j )\otimes 1_B \otimes\eta\, \ra_{\CalH_A} \\
& =\sum_{i, j=1}^n \la \lambda_i (ca\cdot u_i )\otimes 1_B \otimes\xi \, , \,
\mu_j (b\cdot u_j )\otimes 1_B \otimes\eta\, \ra_{\CalH_A} \\
& =\sum_{i, j=1}^n \la 1_B \otimes\xi \, , \, \pi_{\psi}(\la \lambda_i (ca\cdot u_i ) \, , \, 
\mu_j (b\cdot u_j ) \ra_B )(1_B \otimes\eta)\,\ra_{\CalK_{\psi}} \\
& =\sum_{i, j=1}^n \la V_{\psi}\xi \, , \, \pi_{\psi}(\la \lambda_i (ca\cdot u_i ) \, , \, 
\mu_j (b\cdot u_j ) \ra_B )V_{\psi}\eta \ra_{\CalK_{\psi}} \\
& =\sum_{i, j=1}^n \la \xi \, , \, \psi(\la \lambda_i (ca\cdot u_i ) \, , \, 
\mu_j (b\cdot u_j ) \ra_B )\eta \ra_{\CalK} .
\end{align*}
On the othere hand,
\begin{align*}
& \la \pi_{\phi}(c)(a\otimes\xi\otimes\lambda) \, , \, b\otimes\eta\otimes\mu \ra_{\CalH_{\phi}} \\
& =\la \xi \otimes\lambda \, , \, \phi(a^* c^* b)(\eta\otimes\mu) \ra_{\CalK\otimes\BC^n } \\
& =\la \left[ \begin{array}{c}
\lambda_1 \xi \\
\vdots \\
\lambda_n \xi \end{array}\right] \, , \, 
\left[\psi(\la u_i \, , \, a^* c^* b\cdot u_j \ra_B )\right]_{i, j=1}^n
\left[ \begin{array}{c}
\mu_1 \eta \\
\vdots \\
\mu_n \eta \end{array}\right] \ra_{\CalK\otimes\BC^n} \\
& =\la \left[ \begin{array}{c}
\lambda_1 \xi \\
\vdots \\
\lambda_n \xi \end{array}\right] \, , \, 
\left[ \begin{array}{c}
\sum_{j=1}^n \psi(\la u_1 \, , \, a^* c^* b\cdot u_j  \, \ra_B )\mu_j \eta \\
\vdots \\
\sum_{j=1}^n \psi(\la u_n \, , \, a^* c^* b\cdot u_j  \, \ra_B )\mu_j \eta  \end{array}\right] \ra_{\CalK\otimes\BC^n} \\
& =\sum_{i, j=1}^n \la \xi \, , \, \psi(\la \lambda_i u_i \, , \, a^* c^* b\cdot u_j \ra_B )\mu_j \eta \ra_{\CalK} .
\end{align*}
Thus we obtain the conclusion.
\end{proof}

We give the main theorem in the paper.

\begin{thm}\label{thm:Co5} Let $A$ and $B$ be unital $C^*$-algebras, which are strongly Morita
equivalent with respect to an $A-B$-equivalence bimodule $X$. Let $\{u_i \}_{i=1}^n$ be a left
$A$-basis in $X$. Let $\psi$ be a completely positive linear map from $B$ to $\BB(\CalK)$,
where $\CalK$ is a Hilbert space. Let $\phi$ be a map from $A$ to $\BB(\CalK\otimes\BC^n )$
defined by
$$
\left[ \phi(a)_{ij}\right]_{i, j=1}^n =\left[ \psi(\la u_i \, , \, a\cdot u_j \ra_B )\right]_{i, j=1}^n
$$
for any $a\in A$. Then $\phi$ is a completely positive linear map from $A$ to $\BB(\CalK\otimes\BC^n )$,
which is strongly Morita equivalent to $\psi$.
\end{thm}
\begin{proof}
By Lemma \ref{lem:Co1}, we can see that $\phi$ is a completely positive linear map from $A$
to $\BB(\CalK\otimes\BC^n )$. Also, by Lemmas \ref{lem:Co2}, \ref{lem:Co3} and \ref{lem:Co4} and the
definition of strong Morita equivalence for completely positive maps, we can see that $\phi$ is strongly
Morita equivalent to $\psi$.
\end{proof}

\begin{cor}\label{cor:Co6} Let $A$ and $B$ be unital $C^*$-algebras, which are strongly Morita
equivalent. Then there is a $1-1$ correspondence between the set of all strong Morita equivalence
classes of completely positive linear maps on $A$ and the set of all strong Morita equivalence
classes of completely positive linear maps on $B$.
\end{cor}
\begin{proof} This is immediate by Theorem \ref{thm:Co5}.
\end{proof}


\begin{thebibliography}{99}

%\bibitem{AEE:crossed}B. Abadie, S. Eilers and R. Exel,
%{\it Morita equivalence for crossed products by Hilbert $C^*$-bimodules},
%Trans. Amer. Math. Soc.,
%{\bf 350}
%(1998),
%3043--3054.

%\bibitem{Blackadar:K-Theory}B. Blackadar, {\it K-theory for operator algebras},
%M. S. R. I. Publications 5, 2nd Edition,
%Cambridge Univ. Press, Cambridge, 1998.

\bibitem{Blackadar:operator} B. Blackadar, {\it Operator algebras. Theory of $C^*$-algebras and von
Neumann algebras}, Encyclopaedia of Mathematical Sciences,
{\bf 122}, Operator algebras and Non-commutative Geometry III, Springer-Verlag, Berlin, 2006.

%\bibitem{BC:discrete}E. B\'edos and R. Conti,
%{\it On discrete twisted $C^*$-dynamical systems, Hilbert $C^*$-modules and regularity},
%preprint, arXiv: 1104.173lv1.

%\bibitem{BCM:crossed}R. J. Blattner, M. Cohen and S. Montgomery,
%{\it Crossed products and inner actions of Hopf algebras},
%Trans. Amer. Math. Soc.,
%{\bf 298}
%(1986),
%671--711.


%\bibitem{Brown:hereditary}L. G. Brown,
%{\it Stable isomorphism of hereditary subalgebra of $C^*$-algebras},
%Pacific J. Math.,
%{\bf 71}
%(1977),
%335--348.

\bibitem{BGR:linking} L. G. Brown, P. Green and M. A. Rieffel,
{\it Stable isomorphism and strong Morita equivalence of $C^*$-algebras},
Pacific J. Math.
{\bf 71}
(1977),
349--363.

\bibitem{BMS:quasi}L. G. Brown, J. Mingo and N-T. Shen,
{\it Quasi-multipliers and embeddings of Hilbert $C^*$-bimodules},
Can. J. Math.
{\bf 46}
(1994),
1150--1174.

%\bibitem{CK:outer} M. Choda and H. Kosaki,
%{\it Strongly outer actions for an inclusion of factors,}
%J. Func. Anal.
%{\bf 122}
%(1994), 315-332.

%\bibitem{Combes:morita}F. Combes,
%{\it Crossed products and Morita equivalence},
%Proc. London Math. Soc.,
%{\bf 49}
%(1984),
%289--306.

%\bibitem{CMW:equivalence}R. E. Curto, P. S. Muhly and D. P. Williams,
%{\it Cross products of strong Morita equivalent $C^*$-algebras},
%Proc. Amer. Math. Soc.,
%{\bf 90}
%(1984),
%528--530.

\bibitem{ER:multiplier}S. Echterhoff and I. Raeburn,
{\it Multipliers of imprimitivity bimodules and Morita equivalence of
crossed products},
Math. Scand.
{\bf 76}
(1995), 289--309.

%\bibitem{FK:conditional} M. Frank and E. Kirchberg,
%{ \it On  conditional expectations of finite index},
%J. Operator Theory,
%{\bf 40}
%(1998), 87-111.

%\bibitem{JW:covariant}S. Jansen and S. Waldmann,
%{\it The H-covariant strong Picard groupoid},
%J. Pure  Appl. Algebra, 
%{\bf 205}
%(2006), 542--598.

\bibitem{JT:KK}K. K. Jensen and K. Thomsen,
{\it Elements of KK-theory},
Birkh$\ddot a$user,
1991.

%\bibitem{Izumi:simple}M. Izumi,
%{\it Inclusions of simple $C^*$-algebras},
%J. reine angew. Math.,
%{\bf 547}
%(2002), 97--138.

\bibitem{KW1:bimodule}T. Kajiwara and Y. Watatani,
{\it Jones index theory by Hilbert $C^*$-bimodules and K-Thorey},
Trans. Amer. Math. Soc.
{\bf 352}
(2000), 3429--3472.

%\bibitem{KW2:discrete}T. Kajiwara and Y. Watatani,
%{\it Crossed products of Hilbert $C^*$-bimodules by countable discrete groups},
%Proc. Amer. Math. Soc.,
%{\bf 126}
%(1998), 841--851.

%\bibitem{Kodaka:Picard}K. Kodaka,
%{\it Picard groups of irrational rotation $C^*$-algebras},
%J. London Math. Soc.,
%{\bf 56}
%(1997), 179-188.

%\bibitem{Kodaka:equivariance}K. Kodaka,
%{\it Equivariant Picard groups of $C^*$-algebras with finite dimensional $C^*$-Hopf algebra coactions},
%Rocky Mountain J. Math.,
%{\bf 47}
%(2017), 1565-1615.

%\bibitem{Kodaka:generalized}K. Kodaka,
%{\it The generalized Picard groups for finite dimensional $C^*$-Hopf algebra coactions
%on unital $C^*$-algebras},
%preprint, arXiv:1805.08358.

%\bibitem{Kodaka:Picard2}K. Kodaka,
%{\it The Picard groups for unital inclusions of unital $C^*$-algebras},
%Acta Sci. Math. (Szeged),
%\bf
%86
%\rm (2020), 183-207.

%\bibitem{Kodaka:bundle}K. Kodaka,
%{\it Equivalence bundles over a finite group and strong Morita equivalence for unital inclusions
%of unital $C^*$-algebras}, preprint,
%arXiv: 1905.10001.


%\bibitem{Kodaka:countable}K. Kodaka,
%{\it Strong Morita equivalence for inclusions of $C^*$-algebras induced by twisted actions of
%a countable discrete group}, preprint, arXiv: 1910.1003774.

%\bibitem{Kodaka:Picard3}K. Kodaka,
%{\it The Picard groups of unital inclusions of unital $C^*$-algebras induced by involutive
%equivalence bundles},
%preprint.

%\bibitem{Kodaka:bimodule}K. Kodaka,
%{\it Bimodule maps from a unital $C^*$-algebra to its $C^*$-subalgebra and strong Morita equivalence}, preprint,
%arXiv: 2001.10136, 

\bibitem{Kodaka:Picard3}K. Kodaka,
{\it The Picard groups for conditional expectations}, preprint,
arXiv: 2009.08252, 


%\bibitem{KT1:inclusion}K. Kodaka and T. Teruya,
%{\it Inclusions of unital $C^*$-algebras of index-finite type with depth 2 induced by saturated
%actions of finite dimensional $C^*$-Hopf algebras},
%Math. Scand.,
%{\bf 104}
%(2009),
%221--248.

%\bibitem{KT2:coaction}K. Kodaka and T. Teruya,
%{\it The Rohlin property for coactions of finite dimensional $C^*$-Hopf algebras on
%unital $C^*$-algebras},
%J. Operator Theory,
%{\bf 74}
%(2015),
%329--369.

%\bibitem{KT3:equivalence}K. Kodaka and T. Teruya,
%{\it The strong Morita equivalence for coactions of a finite dimensional $C^*$-Hopf algebra on
%unital $C^*$-algebras},
%Studia Math.,
%{\bf 228}
%(2015),
%259--294.

\bibitem{KT4:morita}K. Kodaka and T. Teruya,
{\it The strong Morita equivalence for inclusions of $C^*$-algebras and conditional expectations for
equivalence bimodules}, J. Aust. Math. Soc.
{\bf 105}
(2018), 103--144.

%\bibitem{KT5:inclusion2}K. Kodaka and T. Teruya,
%{\it Coactions of a finite dimensional $C^*$-Hopf algebra on unital
%$C^*$-algebras, unital inclusions of unital $C^*$-algebras and the strong Morita equivalence},
%preprint, arXiv:1706.09430.

%\bibitem{KT6:free} K. Kodaka and T. Teruya,
%{\it Free coactions of a finite dimensional $C^*$-Hopf algebra and strong Morita equivalence}, preprint.

%\bibitem{Lance:toolkit}E. C. Lance, {\it Hilbert $C^*$-modules},
%A toolkit for operatros algebraists, London Math. Soc., Lecture Note Series,
%{\bf 210}, 
%Cambridge Univ. Press, Cambridge, 1995.

%\bibitem{MT:Kac}T. Masuda and R. Tomatsu,
%{\it Classification of minimal actions of a compact Kac algebra with the amenable dual},
%J. Funct. Anal.,
%{\bf 258}
%(2010),
%1965-2025.

%\bibitem{OKT:rohlininclusion}H. Osaka, K. Kodaka and T. Teruya,
%{\it The Rohlin property for inclusions of $C^*$-algebras with a finite Watatani index},
%Operator structures and dynamical systems, 177--195, Contemp Math.,
%{\bf 503}
%%Amer Math. Soc., Providence, RI, 2009.

%\bibitem{Packer:projective}J. A. Packer,
%{\it $C^*$-algebras generated by projective representations of the discrete Heisenberg group},
%J. Operator Theory,
%{\bf 18}
%(1987),
%41--66.

\bibitem{Paulsen:CB-maps} V. Paulsen,
{\it Completely bounded maps and operator algebras},
Cambridge University Press, 2002.

\bibitem{Pedersen:auto} G. K. Pedersen,
{\it $C^* $-algebras and their automorphism groups,}
Academic Press, London, New York, San Francisco, 1979.

%\bibitem{RW:continuous}I. Raeburn and D. P. Williams,
%{\it Morita equivalence and continuous -trace $C^*$-algebras},
%Mathematical Surveys and Monographs, {\bf 60}, Amer. Math. Soc., 1998.

%\bibitem{Rieffel:rotation}M. A. Rieffel,
%{\it $C^*$-algebras associated with irrational rotations},
%Pacific J. Math.,
%{\bf 93}
%(1981),
%415--429.

\bibitem{Stine:positive} W. F. Stinespring,
{\it Positive functions on $C^*$-algebras}, Proc. Amer. Math. Soc.,
{\bf 6}
(1955), 211--216.

\bibitem{Stormer:positive} E. St\o rmer,
{\it Positive linear maps of operator algebras},
Springer-Verlag, Berlin, Heidelberg, 2013.

%\bibitem{Sweedler:Hopf}
%%M. E. Sweedler,
%{\it Hopf algebras},
%Benjamin, New York, 1969.

%\bibitem{Szymanski:subfactor}W. Szyma\'nski,
%{\it Finite index subfactors and Hopf algebra crossed products},
%Proc. Amer. Math. Soc.,
%{\bf 120}
%(1994),
%519--528.

%\bibitem{SP:saturated}W. Szyma\'nski and C. Peligrad,
%{\it Saturated actions of finite dimensional Hopf {\rm *}-algebras on  
%$C^*$-algebras},
%Math. Scand.,
%{\bf 75}
%(1994),
%217--239.

%\bibitem{SW:compact} K. Schwieger and S. Wagner,
%{\it Free actions of compact groups on $C^*$-algebras, Part I}, preprint,
%arXiv: 1505,00688v1.

%\bibitem{Tomiyama:projection}J. Tomiyama,
%{\it On the projection of norm one in $W^*$-algebras},
%Japan Acad.,
%{\bf 33}
%(1957),
%608--612.

%\bibitem{Watatani:index}Y. Watatani,
%{\it Index for $C^*$-subalgebras},
%Mem. Amer. Math. Soc.,
%{\bf 424}, Amer. Math. Soc.,
%1990.

%\bibitem{Zarikian:expectation}V. Zarikian,
%{\it Unique expectations for discrete crossed products}, preprint, arXiv: 170709339v1.

\end{thebibliography}
\end{document}